\documentclass[10pt]{amsart}

\usepackage[textsize=footnotesize,color=blue!40, bordercolor=white,textwidth=78]{todonotes}
\usepackage{amsmath,amssymb,amsthm,enumerate,tikz-cd,tikz}

\usepackage{hyperref} 

\usepackage[all]{xy}

\usepackage{mathabx}

 \usepackage{amsthm, amsmath, amssymb, bm, tikz, ifthen, stmaryrd, placeins
 }
\usepackage
[a4paper, margin=1.2in]{geometry}
 \usepackage{amsmath, amssymb, amscd, amsthm, amsfonts, amstext, bbm, 
  amsbsy,  mathrsfs, hyperref,  upgreek, mathtools, stmaryrd, enumitem, doi}

\definecolor{darkblue}{rgb}{0,0,0.68}
\hypersetup{colorlinks=true,
citecolor=black,
linkcolor=black,
anchorcolor=black
}

\usepackage{hyperref}

 \usepackage{amsmath, amssymb, amscd, amsthm, amsfonts, amstext,
  amsbsy,  mathrsfs, hyperref,  upgreek, mathtools, stmaryrd, doi}

\theoremstyle{definition}

\newtheorem{Def}{Definition}[section]
\newtheorem{Fact}[Def]{Fact}
\newtheorem{Thm}[Def]{Theorem}
\newtheorem{Lem}[Def]{Lemma}

\newtheorem{Rem}[Def]{Remark}

\newtheorem{Cor}[Def]{Corollary}
\newtheorem{Claim}{Claim}[Def]
\newtheorem*{Claim*}{Claim}

\newtheorem{Q}[Def]{Problem}

\newtheorem*{Subclaim*}{Subclaim}
\newtheorem{Question}{Question}
\newtheorem*{Question*}{Question}
\newtheorem{Observation}[Def]{Observation}

\newtheorem{Step}{Step}
\newtheorem*{Step*}{Step}

\newtheorem*{Step-2A}{Step 2A}
\newtheorem*{Step-2B}{Step 2B}

\makeatletter
\newcommand{\biggg}[1]{{\hbox{$\left#1\vbox to 20.5pt{}\right.\n@space$}}}
\newcommand{\Biggg}[1]{{\hbox{$\left#1\vbox to 23.5pt{}\right.\n@space$}}}
\newcommand{\bigggg}[1]{{\hbox{$\left#1\vbox to 26.5pt{}\right.\n@space$}}}
\newcommand{\Bigggg}[1]{{\hbox{$\left#1\vbox to 29.5pt{}\right.\n@space$}}}
\newcommand{\biggggg}[1]{{\hbox{$\left#1\vbox to 32.5pt{}\right.\n@space$}}}
\newcommand{\Biggggg}[1]{{\hbox{$\left#1\vbox to 35.5pt{}\right.\n@space$}}}
\newcommand{\bigggggg}[1]{{\hbox{$\left#1\vbox to 38.5pt{}\right.\n@space$}}}
\newcommand{\Bigggggg}[1]{{\hbox{$\left#1\vbox to 41.5pt{}\right.\n@space$}}}
\makeatother

\newcommand{\leWR}{\le_{\text{W}}^{\mathbb{R}}}

\newcommand{\nleWR}{\nleq_{\text{W}}^{\mathbb{R}}}

\newcommand{\one}{\mathbbm1}

\newcommand{\dom}{\mathrm{dom}} 
\newcommand{\ran}{\mathrm{ran}}

\newcommand{\QQ}{\mathbb{Q}}
\newcommand{\ZZ}{\mathbb{Z}}
\newcommand{\RR}{\mathbb{R}}

\DeclareMathOperator{\id}{id}

\newcommand{\ltl}{{<_{\text{\itshape l\kern-0.02em e\kern-0.02em x}}}}

\newcommand{\mapone}{\xi} 
\newcommand{\maptwo}{\zeta} 

\newcommand{\arrowup}{\star} 

\newcommand{\cl}{\mathrm{cl}}

\newcommand{\conv}{\mathrm{Conv}}

\let\oldotimes\otimes
\renewcommand{\otimes}[1][0pt]{%
  \mathrel{\raisebox{#1}{$\scriptstyle \oldotimes$}}%
  } 

\newenvironment{enumerate-(a)}{\begin{enumerate}[label={\upshape (\alph*)}, leftmargin=2pc]}{\end{enumerate}}

\newenvironment{enumerate-(a)-r}{\begin{enumerate}[label={\upshape (\alph*)}, leftmargin=2pc,resume]}{\end{enumerate}}

\newenvironment{enumerate-(A)}{\begin{enumerate}[label={\upshape (\Alph*)}, leftmargin=2pc]}{\end{enumerate}}

\newenvironment{enumerate-(A)-r}{\begin{enumerate}[label={\upshape (\Alph*)}, leftmargin=2pc,resume]}{\end{enumerate}}

\newenvironment{enumerate-(i)}{\begin{enumerate}[label={\upshape (\roman*)}, leftmargin=2pc]}{\end{enumerate}}

\newenvironment{enumerate-(i)-r}{\begin{enumerate}[label={\upshape (\roman*)}, leftmargin=2pc,resume]}{\end{enumerate}}

\newenvironment{enumerate-(I)}{\begin{enumerate}[label={\upshape (\Roman*)}, leftmargin=2pc]}{\end{enumerate}}

\newenvironment{enumerate-(I)-r}{\begin{enumerate}[label={\upshape (\Roman*)}, leftmargin=2pc,resume]}{\end{enumerate}}

\newenvironment{enumerate-(1)}{\begin{enumerate}[label={\upshape (\arabic*)}, leftmargin=2pc]}{\end{enumerate}}

\newenvironment{enumerate-(1)-r}{\begin{enumerate}[label={\upshape (\arabic*)}, leftmargin=2pc,resume]}{\end{enumerate}}

\newenvironment{itemizenew}{\begin{itemize}[leftmargin=2pc]}{\end{itemize}}

\title{Borel subsets of the real line and continuous reducibility}

\author{Daisuke Ikegami}
\address{SIT Research Laboratories, Shibaura Institute of Technology, 
3-7-5 Toyosu, Koutou-ward, 135-8548 Tokyo, Japan}
\email{ikegami@shibaura-it.ac.jp}

\author{Philipp Schlicht}
\address{School of Mathematics, University of Bristol, University Walk, Bristol BS8 1TW, UK}
\email{philipp.schlicht@bristol.ac.uk} 

\author{Hisao Tanaka}
\address{3-1-1 \ Kitanodai, Hachioji-shi, 
Tokyo 192-0913, Japan}
\email{tanakahk@jcom.zaq.ne.jp}

\thanks{
The authors would like to thank the referee for suggesting various improvements of the presentation.
The first and second author would like to thank the European Science Foundation for support through the grants 3925 and 3749 of the INFTY program. The first author would further like to thank the Japan Society for the Promotion of Science (JSPS) for support through the grants with JSPS KAKENHI Grant Number 262269 and 15K17586. The second author was partially supported by DFG-grant LU2020/1-1 and a Marie Sk\l odowska-Curie Individual Fellowship with number 794020 during the revision of this paper.
}

\begin{document}

\maketitle

\begin{abstract} 
We study classes of Borel subsets of the real line $\RR$ such as levels of the Borel hierarchy and the class of sets that are reducible to the set $\QQ$ of rationals, endowed with the \emph{Wadge quasi-order} of reducibility with respect to continuous functions on $\RR$. Notably, we explore several structural properties of Borel subsets of $\RR$ that diverge from those of Polish spaces with dimension zero. Our first main result is on the existence of embeddings of several posets into the restriction of this quasi-order to any Borel class that is strictly above the classes of open and closed sets, for instance the linear order $\omega_1$, its reverse $\omega_1^\star$ and the poset $\mathcal{P}(\omega)/\mathsf{fin}$ of inclusion modulo finite error. As a consequence of this proof, it is shown that there are no complete sets for these classes. We further extend the previous theorem to targets that are reducible to $\QQ$. These non-structure results motivate the study of further restrictions of the Wadge quasi-order. In our second main theorem, we introduce a combinatorial property that is shown to characterize those $F_\sigma$ sets that are reducible to $\QQ$. This is applied to construct a minimal set below $\QQ$ and prove its uniqueness up to Wadge equivalence. We finally prove several results concerning gaps and cardinal characteristics of the Wadge quasi-order and thereby answer questions of Brendle and Geschke. 
\end{abstract}


\setcounter{tocdepth}{1}
\tableofcontents

\section{Introduction}



Given a topological space $(X,\tau)$, a subset $A$ of $X$ is called \emph{Wadge reducible} or simply \emph{reducible} to a subset $B$ of $X$ if there is a continuous function $f\colon X\rightarrow X$ with $A=f^{-1}(B)$. 
The \emph{Wadge quasi-order} for $(X,\tau)$ is then defined by letting $A\leq B$ if $A$ is reducible to $B$. 
In particular the Wadge quasi-order on subsets of the \emph{Baire space} 
of all functions $f\colon\omega\rightarrow \omega$ 
defines an important hierarchy of complexity in descriptive set theory and theoretical computer science. 
It refines for instance the difference hierarchy, the Borel hierarchy and the projective hierarchy. 
Therefore its structure has been an object of intensive research (see e.g. \cite{MR2279651, MR730585,MR2906999}) 
and moreover, many structural results could be extended to other classes of functions (see e.g. \cite{MR1992531, MR2601013, MR2499419, MR2605897, 
MR3164747, MR3308051}).

In this paper, we study the Wadge quasi-order for the real line. 
Among the few previously known facts were a result by Hertling \cite{Hertling_PhD} that it is ill-founded and an observation by Andretta \cite[Example 3]{Andretta_SLO}, Steel \cite[Section 1]{Steel_Analytic} and Woodin \cite[Remark 9.26]{Woodin_2010} that there are at least three 
Borel subsets that are pairwise incomparable. 
Moreover, a general result by the second author \cite[Theorem 1.5]{dimension} shows that there are uncountably many pairwise incomparable Borel sets. 

The focus of the analysis in this paper is on complexity, 
initial segments and minimal sets of the quasi-order. 
We now give an overview by stating some of the main results in the same order as they appear in the paper.

We first provide instances of the complexity of the quasi-order by embedding several posets such as in the following result. 

\begin{Thm} (Theorem \ref{embed Pomega mod fin}) \label{intro embed Pomega mod fin}
$\mathcal{P}(\omega)/\mathsf{fin}$ is reducible to the Wadge quasi-order for the real line restricted to its $D_2({\bf \Sigma}^0_1)$-subsets. 
\end{Thm}  

This implies that $\omega_1$ and its reverse linear order $\omega_1^\star$ can also be embedded by a result of Parovi\v{c}enko \cite{Parovicenko} and thus the Wadge quasi-order is ill-founded in a strong sense. 
Moreover, this non-structure result suggests the question whether there are initial segment of this quasi-order below fixed sets which do not allow such embeddings. 
It is natural to first consider the set of rationals, where the next result 
shows that a similar embedding is possible.

\begin{Thm} (Theorem \ref{embedding of Pomega mod fin below Q}) 
$\mathcal{P}(\omega)/\mathsf{fin}$ is reducible to the Wadge quasi-order for the real line restricted to the sets that are reducible to $\QQ$ by monotone functions. 
\end{Thm}

The proof of Theorem \ref{intro embed Pomega mod fin} is further used 
to show that the familiar complete ${\bf \Sigma}^0_\alpha$-sets do not exists for the real line for any countable $\alpha\geq 2$ in Theorem \ref{no complete sets}. 



\begin{Thm} 
There is no ${\bf \Sigma}^0_\alpha$-complete and no ${\bf \Pi}^0_\alpha$-complete subset of $\RR$ for any countable ordinal $\alpha\geq 2$. 
\end{Thm} 

To study the Wadge quasi-order below $\QQ$, we introduce the following condition. It characterizes $F_\sigma$ sets that are reducible to $\QQ$. 

\begin{Def} (Definition \ref{definition: condition I}) 
A subset of the real line satisfies the condition ($\text{I}_0$) if it contains the end points of any interval that it contains as a subset. Also, it satisfies ($\text{I}_1$) if its complement satisfies  ($\text{I}_0$), and ($\text{I}$) if both ($\text{I}_0$) and ($\text{I}_1$) hold. 
\end{Def}

In the next characterization, 
\emph{non-trivial} means that both the set and its complement are nonempty. 

\begin{Thm} (Theorem \ref{lem:red_Q_I}) 
A non-trivial $F_\sigma$ subset of $\RR$ satisfies $(\text{I})$ if and only if it is reducible to $\QQ$. 
\end{Thm}

This is used to prove the next main result. There we construct a minimal set below $\QQ$ and 
prove that it is uniquely determined up to Wadge equivalence.

\begin{Thm} (Theorem \ref{minimal set}) 
There is a minimal set that is reducible to $\QQ$. 
\end{Thm}

The structure of the poset $\mathcal{P}(\omega)/\mathsf{fin}$ has been intensively studied, for instance with regard to  
the existence of gaps and values of cardinal characteristics (see \cite{Scheepers_Gaps, MR2768685}). 
Inspired by these results, J\"org Brendle asked whether there are gaps in the Wadge quasi-order for the real line. 
This question is answered in the next theorem. 

\begin{Thm} (Theorem \ref{gaps}) 
For all cardinals $\kappa, \lambda\geq 1$ such that there is a $(\kappa,\lambda)$-gap in $\mathcal{P}(\omega)/\mathsf{fin}$, 
there is a $(\kappa,\lambda)$-gap in the Wadge quasi-order on the Borel subsets of the real line. 
\end{Thm}

In the following result, 
let $\ell$ and $\ell^\star$ denote the suprema of the sizes of well-ordered and reverse well-ordered subsets, respectively.
Answering a question of Stefan Geschke, we determine several cardinal characterstics of the Wadge quasi-order for the real line. 

\begin{Thm} (Theorem \ref{cardinal characteristics}) 
The least size of an unbounded family is $\omega_1$, while the least size of a dominating family and the maximal size of a linearly ordered subset equal $2^\omega$. Moreover, $\ell=\ell^\star$ is consistent with each of the following statements: (a) $\ell=\omega_1$ (b) $\ell=2^\omega$ and (c) $\omega_1<\ell<2^\omega$ and $2^\omega$ is arbitrarily large. 
\end{Thm} 

This paper is organized as follows. Section 2 contains some notations that we use throughout the paper. In Section 3, we first construct an embedding of $\mathcal{P}(\omega)/\mathsf{fin}$ and then extend this result to obtain an embedding with targets below $\QQ$. 
In Section 4, we use the proof of the previous result to prove that no ${\bf \Sigma}^0_\alpha$-complete sets exists for any countable ordinal $\alpha\geq 2$. 
In Section 5, we then characterize $F_\sigma$ sets that are reducible to $\QQ$ by the condition (I) and in Section 6, we construct a minimal  set below $\QQ$. 
Section 7 contains some results about gaps and cardinal characteristics. Finally, we show in Section 8 that the structure of the quasi-order becomes much simpler if the class of reductions is enlarged to finite compositions of right-continuous functions.

\section{Notation} 

In this section, we collect some notation that is used throughout the paper. 



\subsection{Sets of reals } 
We will use the following standard notations for subsets of $\RR$. 
We write $\inf_A$, $\sup_A$, $\min_A$, $\max_A$ for the infimum, supremum, minimum and maximum of a set $A$ whenever they are defined. 
An \emph{end point} of a set of reals $A$ is an element of $A$ that is an end point of an open interval that is disjoint from $A$. 
Let $\mathrm{cl}(A)$ denote the \emph{closure}, $\partial A$ the boundary and $\conv(A)$ the \emph{convex hull} of $A$. 
We will further use the following sets in the \emph{difference hierarchy}. 
A set of reals if called 
$D_2({\bf \Sigma}^0_1)$ if it is equal to the intersection of an open and a closed set. 
Moreover, the class of complements of these sets is called $\check{D}_2({\bf \Sigma}^0_1)$. 
The \emph{Baire space} ${}^\omega\omega$ is the set of functions $f\colon \omega\rightarrow \omega$ equipped with the standard topology that is given by the basic open sets $N_t= \{x\in {}^\omega\omega\mid t\subseteq x\}$ for all $t\in {}^{<\omega}\omega$. 
Moreover, any enumeration of a set is assumed to be injective.

We further introduce the following terminology. 
If a subset of $\RR$ and its complement are nonempty, then the set is called \emph{non-trivial}. 
Moreover, a connected component of a set is called \emph{non-trivial} if is is not a singleton.
We call a real number $x$ \emph{approachable} with respect to a set $A$ if it is a limit of both $A$ and its complement from the same side. Moreover, a set is \emph{approachable} with respect to $A$ if this is the case for one of its elements. 
A \emph{component} of a set is a relatively convex, relatively closed and relatively open subset. 
We write $[x,y)^{\arrowup}=[x,y)$ if $x<y$ and $[x,y)^{\arrowup}=(y,x]$ if $y<x$; the definition is analogous for other types of intervals. 
We further say that $A$ and $B$ are homeomorphic if there is an auto-homeomorphism $f$ of $\RR$ with $f(A)=B$. 
Finally, let $\QQ_2=\{\frac{m}{2^n}\mid m,n \in\ZZ,\ n\geq 1\}$ denote the set of dyadic rationals.


\subsection{Functions on the real line} 
We use the following general terminology. 
A function $f$ on $\RR$ is called \emph{increasing} if $f(x)\leq f(y)$ and \emph{decreasing} if $f(y)\leq f(x)$ for all $x\leq y$. Moreover, $f$ is \emph{strictly increasing} or \emph{strictly decreasing} if this holds for strict inequalities. 
We further call a function \emph{monotone} it it is increasing or decreasing and \emph{strictly monotone} if it is strictly increasing or strictly decreasing. 
Moreover, we say that a sequence $\vec{x}=\langle x_i\mid i\in\omega\rangle$ converges to a real number $x$ \emph{from below} or \emph{above} if its elements are eventually strictly below or above $x$. 

We further fix the following notations for the following specific functions that we will use. 
We will write $c_x$ for the constant map on $\RR$ with value $x$. 
The \emph{Cantor subset} $A_c$ \cite[Exercise 3.4]{Kechris_Classical} of the unit interval consists of all values $\sum_{n= 1}^\infty \frac{ i_n}{3^{n}}$, where each $i_n$ is equal to either $0$ or $2$. 
The \emph{Cantor function} $f_c$ is the unique increasing extension to the unit interval of the function that is defined on $A_c$ by letting $$f_c ( \sum_{n= 1}^\infty \frac{ i_n}{3^{n}} ) = \sum_{n= 1}^\infty \frac{i_n}{2^{n}}.$$ 
We further extend this to a monotone function on $\RR$ by letting $f_c(x+n)=f_c(x)+n$ for all $x\in[0,1]$ and $n\in\ZZ$. 

\subsection{Reducibility}
A \emph{quasi-order} is a transitive reflexive relation. 
The Wadge quasi-order for subset of $\RR$ is defined as follows. 
If $A$ and $B$ are subsets of $\RR$, we say that $A$ is \emph{reducible} to $B$ if there is a continuous function $f\colon \RR\rightarrow \RR$ with $x\in A$ if and only if $f(x)\in B$ for all $x\in \RR$. 
We further write $A\leq B$ for this notion and $A<B$ for \emph{strict reducibility} in the sense that $A\leq B$ but not conversely. 
We finally say that these sets are \emph{Wadge equivalent} if each is reducible to the other set. 
It is then clear that the relation $\leq $ is a quasi-order on the class of subsets of $\RR$. 
A non-trivial set $A$ is called \emph{minimal} if it is minimal with respect to reducibility among all non-trivial sets. 
Finally, the \emph{semi-linear ordering principle} $\mathsf{SLO}$ for a relation $\preceq$ on a class of subsets of an ambient set $X$ states that for all sets $A$ and $B$ in this class, $A\preceq B$ or $B\preceq X\setminus A$. 

\subsection{Functions on quasi-orders and linear orders}

A \emph{reduction} between quasi-orders is a function $f$ with the property that $x\leq y$ if and only if $f(x)\leq f(y)$ for all $x$ and $y$. 
If $X$ and $Y$ are nonempty subsets of a quasi-order, then the pair $(X,Y)$ is a \emph{gap} 
if $x< y$ for all $x\in X$ and $y\in Y$ and there is no $z$ 
with $x< z< y$ for all $x\in X$ and $y\in Y$. 
A \emph{linear order} is an antisymmetric linear quasi-order. 
We will identify any linear order $(L,\leq)$ with its underlying set $L$. 
It is called \emph{scattered} if it does not have a suborder that is isomorphic to the rationals. 
A subset $A$ of $L$ is \emph{convex} if any element of $L$ between elements of $A$ is also an element of $A$. 
We will further consider the linear order $L^{\otimes n}$ that denotes the product of 
$L$ with a set of $n$ elements and can be identified with the linear order that is obtained by replacing every element of $L$ by $n$ elements. 
Moreover, we will denote the lexicographic order on ${}^\omega\omega$ by $\leq_{\mathrm{lex}}$.

All functions with domain $L$ are assumed to be order preserving or order reversing and in particular injective. 
Such a function $\mapone\colon L\rightarrow\mathbb{R}$ is called \emph{discrete} if $\sup_{\mapone(A)}<\inf_{\mapone(B)}$ for all gaps $(A,B)$ in $L$. 
In particular, we will consider functions $\mapone\colon L^{\otimes n}\rightarrow \RR$ 
and will write $\xi_k(i)$ for $\xi(i,k)$, which denotes the $k$-th element of the $i$-th copy of $n$ for any $k<n$ in the above identification. 
Moreover, we define a successor of $\mapone_1(i)$ by letting $\mapone_1^+(i)=\inf_{j>i}\mapone_0(j)$ if $\mapone$ is order preserving and $\mapone_1^+(i)=\sup_{j>i}\mapone_0(j)$ if $\mapone$ is order reversing, where $\inf(\emptyset)=\infty$ and $\sup(\emptyset)=-\infty$. 
For notational convenience, we will further identify $L$ with the subset of $L^{\otimes n}$ that is the product with the first of $n$ elements and can thus write $\mapone(L)$ for its image.


\section{Embedding posets into the Wadge quasi-order} \label{section embedding posets} 


In this section, we construct a reduction of the quasi-order $\mathcal{P}(\omega)/\mathsf{fin}$ of inclusion up to finite error $\subseteq_{\mathsf{fin}}$ on sets of natural numbers to the Wadge quasi-order and thereby generalize \cite[Theorem 5.1.2]{Ikegami_PhD}. 
It is preferable to work with this quasi-order instead of its quotient poset in this setting, since the former allows the construction of definable reductions, while for the latter this would necessitate the use of a definable selector. 
In the next definition, we consider certain unions of intervals -- any continuous reduction will induce certain symmetries in the sequence of intervals and thus the choices of half-open and closed intervals prevent certain reductions. 
From now on, all reductions are assumed to be continuous and $\xi\colon L\rightarrow \RR$ will denote a discrete function on a linear order.

\begin{Def} \label{definition - L-structured}
A subset of $\RR$ is \emph{$L$-structured} by $\mapone\colon L^{\otimes n}\rightarrow\RR$ if it is equal to $\bigcup_{i\in L} A_i$, where $n\geq2$ and each $A_i$ is the union of $[ \mapone_0(i),\mapone_1(i))^{\arrowup}$ with a compact subset $C_i$ of $(\mapone_1(i),\mapone_1^+(i))^{\arrowup}$. 
\end{Def} 

It is possible that there are many $L$-structuring maps $\mapone$ for the same set $A$, but the range of any such map is uniquely determined by the half-open sub-intervals of $A$. 
Moreover, it is important that this structure is preserved under images of reductions. 


\begin{Lem} \label{images of L-structured sets} 
The class of $L$-structured subsets of $\RR$ is stable under images with respect to reductions on $\RR$. 
Moreover, if $f$ reduces $A$ to $B$ and $A$ is $L$-structured by $\mapone$, then $f(A)$ is $L$-structured by $f\circ\mapone$. 
\end{Lem} 
\begin{proof}
Suppose that $A$ is $L$-structured by $\mapone\colon L^{\otimes 2}\rightarrow\RR$, as witnessed by a sequence $\langle C_i\mid i\in L\rangle$ of compact subsets of $(\mapone_1(i),\mapone_1^+(i))^{\arrowup}$. 
Let $i$, $j$ and $k$ denote arbitrary elements of $L$. 
We first claim that $[\mapone_0(i), \mapone_1(i))^{\arrowup}$ is mapped onto $\big[f(\mapone_0(i)), f\bigl(\mapone_1(i)) \big)^{\arrowup}$. 
Since $\mapone_0(i)$ and $\mapone_1(i)$ are elements of $\partial A$, their images lie in the boundary of the interval $f([\mapone_0(i),\mapone_1(i))^{\arrowup})$ -- since $f$ reduces $A$ to $B$, this interval contains $f(\mapone_0(i))$, but not $f(\mapone_1(i))$ and the statement follows. 
To prove the statement of the lemma, we first show that the order and the alignment of the images of the intervals $[\mapone_0(i),\mapone_1(i))^{\arrowup}$ is preserved by the reduction. 

\begin{Claim*} 
If $i<j$, then 
$$f(\mapone_0(i)) \blacktriangleleft f(\mapone_1(i))\blacktriangleleft f(\mapone_0(j)\blacktriangleleft f(\mapone_1(j)),$$ 
where $\blacktriangleleft$ denotes either $<$ or $>$. 
\end{Claim*} 
\begin{proof} 
We choose $\blacktriangleleft$ as either $<$ or $>$ so that $f(\mapone_0(i)) \blacktriangleleft f(\mapone_1(i))$ holds. 
Moreover, we will assume that $\blacktriangleleft$ is equal to $<$ for ease of notation. 

We first show that 
$f(\mapone_1(i))\blacktriangleleft f(\mapone_0(j)$. 
Since $f$ reduces $A$ to $B$, we have $f(\mapone_1(i) ) \neq f (\mapone_0(j))$ and hence we assume $f(\mapone_0(j))\blacktriangleleft f(\mapone_1(i))$ towards a contradiction. 
By our assumption and the fact that $f(\mapone_0(j))$ lies in $\partial B$, we have $f(\mapone_0(j))\leq f(\mapone_0(i))$. 
We now choose some element $x_0$ of $\bigl(f(\mapone_0(i)) , f(\mapone_1(i)) \bigr)^{\arrowup}$. 
Since $f$ is continuous, there is a least element $x$ of $[\mapone_1(i), \mapone_0(j)]^{\arrowup}$ with $f(x) = x_0$. 

We first assume that there is some $k\in L$ with $x\in C_k$. Let $C$ denote the connected component of $x$ in $C_k$ --  since $C_k$ is compact, this is a closed interval or a singleton --  and let $y$ be its least element. 
Thus $f(C)$ is an interval or a singleton and $f(y)\in \partial A$. Since $f(C)$ also contains $x_0$, it is a subset of  $\big[f(\mapone_0(i)), f\bigl(\mapone_1(i)) \big)^{\arrowup}$ and hence $f(y)=f(\mapone_0(i))$. 
By the intermediate value theorem for $\mapone_1(i)$ and $y$, some element  of $(\mapone_1(i),y)$ is mapped to $x_0$, but this contradicts the minimality of $x$, since $y\leq x$. 

We can hence assume that $x\in [\mapone_0(k), \mapone_1(k))^{\arrowup}$ for some $k\in L$ with $i\leq k<j$. Then $ [\mapone_0(k), \mapone_1(k))^{\arrowup}$ is mapped onto $[f(\mapone_0(k)), f(\mapone_1(k))))^{\arrowup}$ --  since $f(x)=x_0$, this interval is equal to $\big[f(\mapone_0(i)), f\bigl(\mapone_1(i)) \big)^{\arrowup}$ and we have $f(\mapone_0(k))=f(\mapone_0(i))$. Now by the intermediate value theorem for $\mapone_1(i)$ and $\mapone_0(k)$, some element  of $(\mapone_1(i),\mapone_0(k))^{\arrowup}$ is mapped to $x_0$, but this contradicts the minimality of $x$.

It remains to show that 
$f(\mapone_0(j))\blacktriangleleft f(\mapone_1(j))$. 
Since $f$ reduces $A$ to $B$, we have $f(\mapone_0(j) ) \neq f (\mapone_1(j))$ and we will thus assume that $f(\mapone_1(j))\blacktriangleleft f(\mapone_0(j))$. 
We choose an arbitrary element $x_0$ of $\bigl(f(\mapone_1(j)) , f(\mapone_0(j)) \bigr)^{\arrowup}$. 
By the intermediate value theorem for $\mapone_1(i)$ and $\mapone_0(j)$, there is a least element $x$ of $[\mapone_1(i), \mapone_0(j)]^{\arrowup}$ with $f(x) = x_0$. 

We first assume that there is some $k\in L$ with $x\in C_k$, let $C$ denote the connected component of $x$ in $C_k$ and $y$ its least element. 
Since the interval $f(C)$ also contains $x_0$, it is a subset of  $\bigl( f(\mapone_1(j)), f(\mapone_0(j))\bigr] ^{\arrowup}$ and hence $f(y)=f(\mapone_0(j))$. 
By the intermediate value theorem for $\mapone_1(i)$ and $y$, some element  of $(\mapone_1(i),y)$ is mapped to $x_0$, but since $y\leq x$ this contradicts the minimality of $x$. 

We finally assume that $x\in [\mapone_0(k), \mapone_1(k))^{\arrowup}$ for some $k\in L$ with $i\leq k<j$. Then $ [\mapone_0(k), \mapone_1(k))^{\arrowup}$ is mapped onto $[f(\mapone_1(j)), f(\mapone_0(j)))^{\arrowup}$ and $f(\mapone_0(k))=f(\mapone_0(j))$. By the intermediate value theorem for $\mapone_1(i)$ and $\mapone_0(k)$, some element  of $(\mapone_1(i),\mapone_0(k))^{\arrowup}$ is mapped to $x_0$, but this contradicts the minimality of $x$. 
\end{proof} 

It remains to show that the image of $A$ is $L$-structured -- 
this is witnessed by the sets $D_i=f(C_i)\cap (f(\mapone_1(i),f(\mapone_1^+(i))^{\arrowup}$ by the next claim. 


\begin{Claim*} 
If $i<j$, then 
$$f(\mapone_0(i))\blacktriangleleft f(\mapone_1(i))\blacktriangleleft f(C_i) \blacktriangleleft f(\mapone_1(j))$$ 
and  if $i$ has no direct successor, 
then $f(C_i)\blacktriangleleft f(\mapone_1^+(i))$, where $\blacktriangleleft$ denotes either $<$ or $>$. 
\end{Claim*} 
\begin{proof} 
We choose $\blacktriangleleft$ as either $<$ or $>$ so that $f(\mapone_0(i)) \blacktriangleleft f(\mapone_1(i))$ holds. Moreover, we will assume that $\blacktriangleleft$ is equal to $<$ for ease of notation. 

We first show that 
$f(\mapone_1(i))\blacktriangleleft f(C_i)$. 
Otherwise let $x$ denote the least element of $C_i$ with $f(x)\leq f(\mapone_1(i))$. Moreover, let $C$ denote the connected component of $x$ in $C_i$ and $y$ its least element -- since $f$ reduces $A$ to $B$, it follows that $f(C)\blacktriangleleft f(\mapone_1(i))$ and moreover $f(y)\leq f(\mapone_0(i))$, since $f(y)\in \partial B$. 
By continuity, some element of $(\mapone_1(i),y)$ is mapped to $(f(\mapone_0(i)),f(\mapone_1(i)))^{\arrowup}$. Since $f$ reduces $A$ to $B$, this would have to be an element of $C_i$, contradicting the minimality of $x$.

We now claim that 
$f(C_i) \blacktriangleleft f(\mapone_1(j))$. 
If this fails, let $x$ denote the least element of $C_i$ with $f(\mapone_1(j))\leq f(x)$ -- since $f$ reduces $A$ to $B$, it follows that $f(\mapone_1(j))\blacktriangleleft f(x)$. We choose an arbitrary $y_0\in (\mapone_0(j),\mapone_1(j))^{\arrowup}$. 
By the intermediate value theorem for $\mapone_1(i)$ and $x$, some element $y$ of $(\mapone_1(i),x)^{\arrowup}$ is mapped to $y_0$ and we assume that $y$ is maximal. Let $C$ be the connected component of $y$ in $C_i$ and let $z$ be its maximum. Since $f$ reduces $A$ to $B$, it follows that $f(z)= \mapone_0(j)$. 
By the intermediate value theorem for $z$ and $x$, some element of $(z,x)^{\arrowup}$ is mapped to $y_0$. Since $y\leq z$, this contradicts the maximality of $y$.

It remains to show that 
$f(C_i)\blacktriangleleft f(\mapone_1^+(i))$ if $i$ has no direct successor. 
Otherwise let $x$ denote the least element of $C_i$ with $f(\mapone_1^+(i))\leq f(x)$. Since $i$ has no direct successor, it follows from the preservation of the order of the intervals proved in the first claim that there is some $x_0\in B$ with $f(\mapone_1^+(i))\blacktriangleleft x_0\blacktriangleleft f(x)$. By the intermediate value theorem for $\mapone_1^+(i)$ and $x$, some element of $(\mapone_1^+(i),x)^{\arrowup}$ is mapped to $x_0$, contradicting the minimality of $x$. 
\end{proof} 
If $\blacktriangleleft$ is equal to $>$ instead of $<$, we replace $\leq$ with $\geq$ and choose maximal instead minimal elements and conversely, but the rest of the proof remains unchanged. 
Thus $f(A)$ is $L$-structured by $f\circ \mapone$ as witnessed by the sequence $\langle D_i\mid i\in L\rangle$, where  $D_i=f(C_i)\cap (f(\mapone_1(i),f(\mapone_1^+(i))^{\arrowup}$. 
\end{proof}

In the following, we will use the fact that every countable scattered linear order can be discretely embedded into the real line. This can be easily shown by induction on the Cantor-Bendixson rank, using the fact that any such linear order has countable rank by a result of Hausdorff. 
We will use the following sets in the proofs below. 
Suppose that $L$ is a countable scattered linear order and $\mapone \colon L^{{\scriptscriptstyle\otimes}4}\rightarrow \mathbb{R}$ is discrete. 
For any subset $b$ of $L$, we define $A_b$ as the set 
$$ \bigcup\nolimits_{i\in L} [\mapone_0(i),\mapone_1(i))^{\arrowup} \cup [\mapone_2(i),\mapone^b(i)]^{\arrowup},$$ 
where $\mapone^b(i)=\mapone_2(i)$ if $i\in b$ and $\mapone^b(i)=\mapone_3(i)$ otherwise.  
Thus these sets are $L$-structured and by our assumption that $\mapone$ is a discrete embedding, they are $D_2({\bf \Sigma}^0_1)$-sets. 


The next result shows that any continuous reduction between sets of the form $A_b$ induces an isomorphism of $L$ onto a convex suborder of $L$, and conversely. For the statement of the result, we introduce the following notation. If $\sigma \colon L\rightarrow L$ is a function, let $\sigma_\mapone\colon \mapone(L)\rightarrow\mapone(L)$ denote the function $\mapone\circ \sigma\circ \mapone^{-1}$. Moreover, if $f\colon\mathbb{R}\rightarrow\mathbb{R}$ is a function with $f(\mapone(L)) \subseteq \mapone(L)$, let $f^\mapone\colon L\rightarrow L$ denote the function $\mapone^{-1}\circ f\circ \mapone$ (see Figure \ref{figure 1}). 



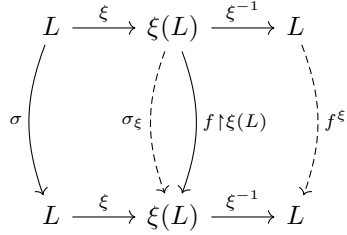
\begin{figure}[h]
\begin{center} 
\begin{tikzcd}
  L\arrow[r, "\mapone"] \arrow[ddd, "\sigma", swap, bend right=20] 
     & \mapone(L) \arrow[r, "\mapone^{-1}"] \arrow[ddd, "\sigma_\mapone", bend right=20, dashed, swap] \arrow[ddd, "f{\upharpoonright}\mapone(L)", bend left=20] 
     & L \arrow[ddd, "f^\mapone", dashed, bend left=20] \\ 
   &  
   & 
    \\ 
      &  \\ 
  L \arrow[r, "\mapone"] 
     & \mapone(L) \arrow[r, "\mapone^{-1}"] 
     & L \\ 
\end{tikzcd}
\end{center} 
\caption{Relationship between reductions and order isomorphisms. } 
\label{figure 1} 
\end{figure}

\begin{Lem} \label{lem:discrete_emb} 
Given a fixed discrete embedding $\mapone\colon L^{\otimes 2}\rightarrow \RR$,  
the following correspondences hold for all subsets $a$ and $b$ of $L$. 
\begin{enumerate-(1)} 
\item \label{lem:discrete_emb 1} 
\begin{enumerate-(a)} 
\item \label{lem:discrete_emb 1a} 
If $\sigma$ is an isomorphism of $L$ onto a convex subset with $\sigma(a)\subseteq b$, then $\sigma_\mapone$ can be extended to an increasing reduction $f$ of $A_a$ to $A_b$. 
\item \label{lem:discrete_emb 1b} 
If $f$ is a reduction of $A_a$ to $A_b$, then $f^\mapone$ is a well-defined isomorphism of $L$ onto a convex subset with $f^\mapone(a)\subseteq b$.  
\end{enumerate-(a)} 
\item \label{lem:discrete_emb 2} 
\begin{enumerate-(a)} 
\item \label{lem:discrete_emb 2a} 
If $\sigma$ is an automorphism of $L$ with $\sigma(a)\subseteq b$, then $\sigma_\mapone$ can be extended to an increasing surjective reduction $f$ of $A_a$ to $A_b$. 
\item \label{lem:discrete_emb 2b} 
If $f$ is a surjective reduction of $A_a$ to $A_b$, then $f^\mapone$ is a well-defined automorphism of $L$ with $f^\mapone(a)\subseteq b$. 
\end{enumerate-(a)} 
\end{enumerate-(1)} 
\end{Lem} 
\begin{proof} 
We will write $\ran_a(\mapone)$ for the set of all $\mapone_k(i)$ for $k\in \{0,1,2\}$ and all $\mapone^a(i)$ for $i\in L$. 
It is sufficient to prove the first claim, since the second claim follows immediately.

We begin by assuming that $\sigma$ is given and defining the following map on 
$\ran_a(\mapone)$. 
Let $f(\mapone_k(i))=\mapone_k(\sigma(i))$ for $k\in\{0,1,2\}$ and $f(\mapone^a(i))=\mapone^b(\sigma(i))$ for $i\in L$ -- this is well-defined since $\sigma(a)\subseteq b$. 
Moreover, this map is increasing and has a unique continuous extension to the closure $C$ of $\ran_a(\mapone)$, since $\ran_a(\mapone)$ is scattered. 
Since $\xi$ is discrete by our assumption, this can be easily extended to a map on $\RR$ with the required properties, for instance by choosing $f$ to be affine on each maximal open interval that is disjoint from $C$.

Now assume that $f$ is given. By Lemma \ref{images of L-structured sets}, the image $f(A_a)$ is $L$-structured by $f\circ \mapone$. Since $\ran(f)$ is a convex subset of $\RR$, there is a convex subset $L_0$ of $L$ with $\ran(\mapone{\upharpoonright}L_0)=\ran(f\circ \mapone{\upharpoonright}L)$. Hence $f^\xi=\xi^{-1}\circ f\circ \xi$ is an isomorphism from $L$ to $L_0$. 
Moreover, assume that $\langle C_i\mid i\in L\rangle$ and $\langle D_i\mid i\in L\rangle$ witness that $f(A_a)$ is $L$-structured by $f\circ \mapone$ and $A_b$ is $L$-structured by $\mapone$. Since $C_i=D_{f_\mapone(i)}$ and this is a singleton for all $i\in a$, it follows that $f_\mapone(i)\in b$. 
\end{proof} 

We note that if the singletons are omitted in the definition of the sets $A_b$, it can be seen that we obtain a version of the previous lemma for functions that are not necessarily increasing. 

In the next theorem, we consider the set $\ZZ^{(\omega)}=\{\vec{z}=\langle z_i\mid i\in\omega\rangle \in \ZZ^\omega\mid \exists n_0\ \forall n\geq n_0\ z_n=0\}$ equipped with the anti-lexicographical order that is defined by $\vec{x}=\langle x_i\mid i\in\omega\rangle\leq\vec{y}=\langle y_i\mid i\in\omega\rangle$ if $\vec{x}=\vec{y}$ or $x_i<y_i$ for the largest $i$ with $x_i\neq y_i$. It can be easily seen that this linear order is scattered -- it is obtained as the direct limit of $\langle \ZZ^n\mid n\in\omega\rangle$, where $\ZZ^{n+1}$ is ordered as a sequence with order type $\ZZ$ of copies of $\ZZ^n$. Moreover, it can be easily checked that every isomorphism of $\ZZ^{(\omega)}$ onto a convex subset is surjective. 
For any subset $a$ of $\omega$, let $a^\star=\{\vec{z}\in \ZZ^{(\omega)}\mid \vec{z}\leq 0^{\omega} \vee \exists n\in a\ \vec{z}=0^n{}^\smallfrown 1{}^\smallfrown 0^\omega\}$ 
and thus $a$ is coded into a sequence of elements of $\ZZ^{(\omega)}$ that are spaced further and further apart. 
Moreover, let $\iota_n$ denote the shift of $\ZZ^{n+1}$ in its last coordinate that is extended to $\ZZ^{(\omega)}$ by the identity on the remaining copies of $\ZZ^{n+1}$. 
I.e. $\iota_n$ is defined by $\iota_n(\vec{x})=\vec{y}$ for $\vec{x}=\langle x_i\mid i\in\omega\rangle$ and $\vec{y}=\langle y_i\mid i\in\omega\rangle$, where $x_i=y_i$ for all $i\neq n$ and $x_n+1=y_n$ if $x_i=0$ for all $i>n$ and $\vec{x}=\vec{y}$ otherwise. Using this linear order and Theorem \ref{lem:discrete_emb}, we can now prove the first main result.

\begin{Thm} \label{embed Pomega mod fin} 
$\mathcal{P}(\omega)/\mathsf{fin}$ is reducible to the Wadge quasi-order for the real line restricted to its $D_2({\bf \Sigma}^0_1)$-subsets. 
\end{Thm} 
\begin{proof} 
We first claim that for all subsets $a$ and $b$ of $\omega$, $a\subseteq_{\mathsf{fin}}b$ holds if and only if there is an automorphism $\sigma$ of $\ZZ^{(\omega)}$ with $\sigma(a^\star)\subseteq b^\star$. 
To prove this, we first assume that $a\subseteq b\setminus n$ and let $\sigma=\iota_n^{-1}$. 
Then for any $m\in a$, we have $\sigma(\vec{z})\leq 0^\omega$ if  $m\leq n$ and $\sigma(\vec{z})=\vec{z}$ if $m>n$ and hence $\sigma(a^\star)\subseteq b^\star$. 
For the other implication, assume that $\sigma$ is an automorphism of $\ZZ^{(\omega)}$ with $\sigma(a^\star)\subseteq b^\star$. 
Let $\sigma(0^\omega)=\vec{x}=\langle x_i\mid i\in\omega\rangle$ with $x_i=0$ for all $i\geq n$. Then it is easy to see that for any $\vec{y}=\langle y_i\mid i\in\omega\rangle$ and $i\geq n$, we have $\sigma(\vec{y})=\vec{z}=\langle z_i\mid i\in\omega\rangle$ with $y_i=z_i$ for all $i\geq n$. 

To see that $a\subseteq_{\mathsf{fin}}b$, we assume that $m\geq n$ is an element of $a$ and let $\vec{x}=0^m{}^\smallfrown 1 {}^\smallfrown 0^\omega$. Since $\vec{x}=\sigma(\vec{x})\in\sigma(a^\star)\subseteq b^\star$ by assumption, it follows that $m\in b$ as required. 
To prove the statement, we fix any discrete embedding $\mapone\colon L^{\otimes 4}\rightarrow \RR$ for $L=\ZZ^{(\omega)}$ with $\ran(\mapone)$ cofinal and coinitial in $\RR$ and let $A_b$ denote the sets defined before Lemma \ref{lem:discrete_emb} 
for subsets $b$ of $\omega$. 
Since it is easy to see that every isomorphism of $\ZZ^{(\omega)}$ onto a dense subset is an automorphism, Lemma \ref{lem:discrete_emb} now implies that $a\subseteq_{\mathsf{fin}} b$ if and only if $A_{a^\star}$ is reducible to $A_{b^\star}$. Thus the map that sends a subset $b$ of $L$ to $A_{b^\star}$ is a reduction as required. 
\end{proof}

Since every poset of size $\omega_1$ is embeddable into $\mathcal{P}(\omega)/\mathsf{fin}$ by a result of Parovi\v{c}enko \cite{Parovicenko}, the previous result shows in particular that $\omega_1$ and its reverse $\omega_1^\star$ are embeddable into the Wadge quasi-order on the Borel sets. 
Moreover, we will show in Theorem \ref{cardinal characteristics} below that this is optimal in the sense that it is consistent with large  values of the continuum that neither of $\omega_2$ and $\omega_2^\star$ is embeddable. 
We would further like to point out that the above proofs also show that the equivalence relation defined by Wadge equivalence on the collection of countable unions of intervals is not smooth. 
To state this more precisely, we fix a canonical coding of Borel sets and can then easily check that the map that sends a subset $b$ of $L$ to $A_b$ is Borel measurable in the codes. 
By the previous results, this defines a reduction of the equivalence relation $E_0$ of equality up to finite error to Wadge equivalence of countable unions of intervals. 
Moreover, by a standard result in Borel reducibility, this implies that this equivalence relation is not smooth.

Extending Theorem \ref{embed Pomega mod fin}, we finally construct a reduction of $\mathcal{P}(\omega)/\mathsf{fin}$ to the Wadge quasi-order below $\QQ$. 
The construction is based on the proof of Theorem \ref{embed Pomega mod fin}, but instead of adapting the previous more general method, we here present a shorter and more direct proof. 

To this end, we will form preimages of countable dense sets with respect to the Cantor function and thus replace countably many points by closed intervals. 
To define the sets that are used to prove the next result, we first let $x_0\in \RR\setminus \QQ_2$, $A=f_c^{-1}(\QQ_2+x_0)$, $B=f_c^{-1}(\QQ_2)$ and $B^+=B\cup [0,\frac{2}{3}]$. 
We further let $[x,y)^D$ denote the image of $D\cap [0,1)$ by the unique affine increasing map from $[0,1)$ onto $[x,y)$ for any subset $D$ of $\RR$. 
Moreover, we let $L$ denote the ordinal $\omega^\omega$ defined via ordinal exponentiation, fix an order-preserving continuous map $\mapone\colon L^{\otimes 2}\rightarrow \RR$ 
and define 
$$A_b=\bigcup_{i\in L} [\mapone_0(i),\mapone_1(i))^A\cup [\mapone_1(i),\mapone_0(i+1))^{B(i)}$$ 
for any subset $b$ of $\omega$, where $B(i)$ is equal to $B^+$ if $i=\omega^n$ for some $n\notin b$ and to $B$ otherwise. 
It can be easily seen that $A_b$ is reducible to $\QQ$ by a monotone function, since this is true for both $A$ and $B^+$ and moreover $\mapone_0(0)\notin A_b$ by the choice of $x_0$. 

In the next proof, we use the following facts about these sets. 
If $I$ denotes the unit interval, then $B^+\cap I$ is reducible to $B\cap I$ by a function on $I$ that maps $[0,\frac{2}{3}]$ to $0$ while expanding the remainder by a factor of $3$. However, there is no function on $I$ that fixes $0$ and reduces $B\cap I$ to $B^+\cap I$, since the closed subinterval of $B^+$ at $0$ could not appear in the image. 
Moreover, for any open interval $I$ that is not contained in and not disjoint from $A$ and for $C$ either equal to $B$ or $B^+$, the set $A\cap I$ is not reducible to $C$ by a function on $I$ and conversely for $A$ and $C$ exchanged. 
To see this, use the fact that $A$ and $\RR\setminus C$ are nowhere dense. 
Assuming that $f$ is a function on $I$ that reduces $A\cap I$ to $C$, it follows that the image of $I\setminus A$ is somewhere dense but this contradicts the previous facts and moreover, the proof of the reverse case is symmetric by passing to complements. 

\begin{Thm} \label{embedding of Pomega mod fin below Q} 
$\mathcal{P}(\omega)/\mathsf{fin}$ is reducible to the Wadge quasi-order for the real line restricted to the sets that are reducible to $\QQ$ by monotone functions. 
\end{Thm} 
\begin{proof} 
We will show that $a\subseteq_{\mathsf{fin}} b$ if and only if $A_a$ is reducible to $A_b$. To this end, we first assume that $a\subseteq_{\mathsf{fin}}b$ and choose some $n_0$ such that every $n\in a$ with $n\geq n_0$ is an element of $b$. 
Using the fact that there is a function on the unit interval that fixes both $0$ and $1$ and reduces $B^+$ to $B$, we can easily obtain a monotone reduction of $A_a$ to $A_b$ that maps $\mapone_k(i)$ to $\mapone_k(\omega^{n_0}+1+i)$ for all $i\in L$ and $k\leq 1$.

For the reverse implication, assume that $A_a$ is reducible to $A_b$ by a function $f$. 
By the non-reducibility result for $A$ and $B/B^+$ that is stated before this theorem, no element of $(\mapone_0(i),\mapone_1(i))$ is mapped to $(\mapone_1(j),\mapone_0(j+1))$ for any $j$. 
Moreover, since no element of $[\mapone_0(i),\mapone_1(i))^A$ is mapped strictly above or below $\ran(\mapone)$, it is easy to verify that this holds for the whole interval $[\mapone_0(i),\mapone_1(i)]$. 
Therefore each interval $[\mapone_0(i),\mapone_1(i)]$ is mapped to $[\mapone_0(\alpha_i),\mapone_1(\alpha_i)]$ for some $\alpha_i\in L$. 
Moreover, we have $f(\mapone_k(i))=\mapone_k(\alpha_i)$ unless $k=i=0$, since in this case $\mapone_k(i)$ lies on the boundary between the subsets of $A_a$ that are given by $A$ and $B$, respectively. 
We now claim that $\alpha_i=\alpha_0+i$ for all $i\in L$. 
This holds for $i=0$ by the choice of $\alpha_0$ and follows from continuity of $\mapone$ and $f$ in the limit step. 
Moreover, the claim again follows from the non-reducibility result for $A$ and $B/B^+$ in the successor case. 
Assuming that $\alpha_0<\omega^{n_0}$ and using that $\omega^{n_0}$ is additively closed, we thus have $\alpha_i=i$ for all $i\geq \omega^{n_0}$. 
Now every $n\in a$ with $n\geq n_0$ is an element of $b$, since otherwise the closed interval in $A_b$ at $\mapone_1(\omega^n)$ cannot appear in the image of the reduction. 
\end{proof}

\section{No complete sets} 

Recall that a set $A$ in a collection $\mathcal{C}$ that is quasi-ordered by a relation $\preceq$ is called \emph{complete} for $\mathcal{C}$ if $B\preceq A$ holds for all sets $B\in \mathcal{C}$. 
The next result shows that there are no complete sets for the levels of the Borel hierarchy on $\RR$ except on the lowest level. 

\begin{Thm} \label{no complete sets} 
There is no ${\bf \Sigma}^0_\alpha$-complete and no ${\bf \Pi}^0_\alpha$-complete subset of $\RR$ for any countable ordinal $\alpha\geq 2$. 
In fact, there is a collection $\vec{B}=\langle B_a\mid a\subseteq\omega\rangle$ of $D_2({\bf \Sigma}^0_1)$-subsets of $\RR$ such that for all subsets $B$ of $\RR$, we have $B_a\leq B$ for only countably many sets $a$ 
and moreover, the same statement holds for $\check{D}_2({\bf \Sigma}^0_1)$-sets. 
\end{Thm} 
\begin{proof} 
We let $L$ denote the ordinal $\omega^2$ with its reverse order and fix a discrete order-preserving map $\mapone\colon L^{\otimes 2}\rightarrow \RR$. 
For any subset $b$ of $\omega$, let $A_b$ be a subset of $\RR$ that is $L$-structured by $\mapone$, witnessed by a sequence $\vec{C}=\langle C_i^b\mid i\in L\rangle$ of compact sets with $C_i^b$ nonempty if and only if $n\in b$. 
Suppose that there is a subset $B$ of $\RR$ and reductions $f_b$ of $A_b$ to $B$ for uncountably many $b$. 
By Lemma \ref{images of L-structured sets}, we have that $f_b(A_b)$ is $L$-structured by $f_b\circ \mapone$. 
Moreover, by the second claim in the proof of Lemma \ref{images of L-structured sets}, this is witnessed by the sequence $\langle f_b(C_i^b)\mid i\in L\rangle$. Note that $f_b(C_i^b)$ is nonempty if and only if $i\in b$. 
For each such $a$, the image of $[\mapone_0(0),\mapone_1(0))^{\arrowup}$ is equal to the interval $[f_a(\mapone_0(0)),f_a(\mapone_1(0)))^{\arrowup}$, since $f_a$ is a reduction of $A_a$. 
Moreover, there are distinct $a$, $b$ such that the corresponding intervals have non-empty intersection and thus they are equal. 
We may now assume that $\sup(f(A_a))\leq \sup(f(A_b))$ by symmetry. 
Since the sets $f_a(A_a)$, $f_b(A_b)$ are $L$-structured by $f_a\circ \mapone$, $f_b\circ \mapone$ and their first intervals are equal, it follows that $f=(f_b\circ \mapone)^{-1} \circ(f_a\circ \mapone)$ is an isomorphism of $L$ onto a convex subset with $f(0)=0$. Thus $f=\id$ and 
$f_a\circ \mapone = f_b\circ \mapone$. 
Hence for all $i\in L$ and all $n\in\omega$, we have 
$n\in a \Longleftrightarrow C_i^a\neq \emptyset\Longleftrightarrow f_a(C_i^a)= f_b(C_i^b)\neq \emptyset \Longleftrightarrow C_i^b\neq \emptyset \Longleftrightarrow n\in b$,  
contradicting our assumption that $a\neq b$. 
\end{proof}

\section{Restrictions to classes of $F_\sigma$ sets} \label{section Fsigma} 





Beginning with the simplest sets, Selivanov \cite{Selivanov_Variations} observed that any two non-trivial closed subsets of $\mathbb{R}$ are Wadge equivalent and the same holds for open sets by passing to complements; hence these sets are minimal. 
It was further noticed by Steel \cite[Section 1]{Steel_Analytic} that any non-trivial open subset of $\RR$ is not reducible to $\QQ$ and vice versa. 
Progressing to $F_\sigma$ sets, we will consider the following combinatorial conditions with a twofold motivation: first, they characterize incomparability with open and closed sets and second, 
reducibility to $\QQ$.

\begin{Def} \label{definition: condition I} 
A subset of the real line satisfies condition ($\text{I}_0$) (where $\text{I}$ stands for \emph{interval closure}) if it contains the end points of any interval that it contains as a subset. Also, it satisfies ($\text{I}_1$) if its complement satisfies  ($\text{I}_0$), and ($\text{I}$) if both ($\text{I}_0$) and ($\text{I}_1$) hold. 
\end{Def} 

For example, any countable dense set satisfies $(\text{I})$ and it is easy to see that these conditions are stable under forming continuous preimages. 
Note that a set $A$ satisfies $(\text{I}_1)$ if and only if every element of $A$ is a limit of $A$ from both sides. Moreover, the next result characterizes this condition in terms of continuous reducibility. 

\begin{Lem}\label{I} 
The condition $(\text{I}_0)$ for a non-trivial subset $A$ of $\RR$ is equivalent to the statements that no non-trivial open set is reducible to $A$ and that no non-trivial open set is comparable with $A$. 
\end{Lem}

\begin{proof}
We first assume that some non-trivial open set is comparable with $A$. If $A$ is reducible to a open set, then it is itself open and it follows that every non-trivial open set reduces to $A$.  
The same conclusion follows if some open set reduces to $A$. 
In particular, there is a reduction $f$ of $(0,\infty)$ to $A$. Thus $f((0,\infty))$ contains a sub-interval of $A$ with an end point $f(0)$ that is not contained in $A$ and thus witnesses the failure of $(\text{I}_0)$. Assuming conversely that $(\text{I}_0)$ fails, there are $x<y$ such that $(x,y)$ is a subset of $A$ and $x\notin A$ or $y\notin A$ -- we can assume that the former holds. It is then easy to define a reduction of any non-trivial open set $U$ to $A$ that maps $U$ to $(x,y)$ and it complement to $x$. 
\end{proof}

To prove a characterization of $F_\sigma$ sets that satisfy  $(\text{I}_0)$, we will make use of the following property of $F_\sigma$ subsets of the Baire space. 
To state the result, we say that a sequence $\vec{B}$ 
\emph{refines} a sequence $\vec{A}$ 
if every element of $\vec{B}$ is a subset of some element of $\vec{A}$. 
Moreover, let $T/s=\{t\in T\mid s\subseteq t\vee t\subseteq s\}$ if $T$ is a subtree of $\omega^{<\omega}$ and $s\in T$. 


\begin{Fact} \label{lem:Baire space F_sigma} 
For every sequence $\vec{A}$ of closed subsets of the Baire space, there is a sequence $\vec{B}$ of pairwise disjoint closed sets that refines $\vec{A}$. 
\end{Fact} 
\begin{proof} 
Assuming that $\vec{A}=\langle A_n\mid n\in\omega\rangle$, let $T_n$ be a subtree of $\omega^{<\omega}$ with $A_n=[T_n]=\{x\in {}^\omega\omega\mid \forall m\in\omega\ x{\upharpoonright}m\in T_n\}$ for all $n\in\omega$. 
If $\vec{s}=\langle s_{n,i} \mid i\in I_n\rangle$ enumerates the set of minimal nodes in $T_{n}\setminus \bigcup_{m<n}T_m$, 
then we have  
$\bigcup_{m\leq n}[T_m]=\bigcup_{m\leq n, i\in I_m}[T_m/{s_{m,i}}]$ and hence $\bigcup_{n\in \omega}[T_n]=\bigcup_{n\in\omega, i\in I_n}[T_n/{s_{n,i}}]$. 
Moreover, the sets $[T_n/{s_{n,i}}]$ are pairwise disjoint by the choice of $\vec{s}$. 
\end{proof} 

It is easy to see that the previous fact does not hold for all $F_\sigma$ subsets of $\RR$, since it fails for instance for any sequence $\vec{A}$ whose union is an open interval. 
However, one direction of the next theorem shows that the required decomposition is possible if $(\text{I}_0)$ is assumed. To prove this, we need the following fact. 

\begin{Fact} \label{subinterval of union of disjoint closed sets} 
If an open interval $I$ is contained in the union of a sequence $\vec{A}$ of pairwise disjoint closed sets, then it is contained in one of the sets. 
\footnote{This can also be proved by observing that otherwise, the quotient of $I$ with respect to $\vec{A}$ is a regular Hausdorff space and hence metrizable. However, this contradicts the fact that it is also both countable and connected, which is impossible. } 
\end{Fact} 
\begin{proof} 
Since the claim is clear for finite sequences, we assume that $\vec{A}=\langle A_n\mid n\in\omega\rangle$ and further let $A=\bigcup_{n\in\omega} A_n$ and  $A^n=\bigcup_{k\leq n}A_k$ for $n\in\omega$. 
Assuming that the claim fails, we will construct a sequence $\vec{I}=\langle I_k \mid k \in \omega\rangle$ of open subintervals of $I$ and a strictly increasing sequence $\vec{n}=\langle n_k \mid k \in \omega\rangle$ of natural numbers with the following properties by induction. 
\begin{enumerate-(a)}
\item \label{interval in union of closed sets 2} 
$I_k \cap A^{n_k}= \emptyset$ 
\item \label{interval in union of closed sets 4} 
$\cl(I_k)\cap A_{n_k}\neq \emptyset$ 
\item \label{interval in union of closed sets 3} 
$\cl(I_{k+1})\subseteq I_{k}$ 
\end{enumerate-(a)}
Since $I$ is a subset of $A$, there is some $n\in\omega$ with $I\cap A_n \neq \emptyset$ and we let $n_0$ be the least such. Since we assumed that $I\not\subseteq A_n$ for all $n\in\omega$, we have $I\setminus A_{n_0}\neq\emptyset$ and therefore there is a subinterval $I_0$ of $I$ with $\cl(I_0)\subseteq I$ that is disjoint from $A_{n_0}$ with an end point in $A_{n_0}$. 
Now assume that both $I_k$ and $n_k$ are defined. 
Since $I_k$ is contained in $A$ 
and $I_k \cap A^{n_k}= \emptyset$ by \ref{interval in union of closed sets 2}, there is some $n>n_k$ with $I_k \cap A_n  \neq \emptyset$ and we let $n_{k+1}$ be the least such. 
Since one of the end points of $I_k$ is an element of $A_i$ for some $i\leq n_k$ by \ref{interval in union of closed sets 4} and the elements of $\vec{A}$ are pairwise disjoint, 
we have $I_k\setminus A_{n_{k+1}}\neq\emptyset$. 
Hence there is a subinterval $I_{k+1}$ of $I_k$ with $\cl(I_{k+1})\subseteq I_k$ that is disjoint from $A^{n_{k+1}}$ with an end point in $A_{n_{k+1}}$. 
Finally, the intersection of $\vec{I}$ is disjoint from $A$ by \ref{interval in union of closed sets 2} and a nonempty subset of $I$ by \ref{interval in union of closed sets 3}. However, this contradicts the assumption that $I$ is a subset of $A$. 
%
\end{proof}

\begin{Thm}\label{prop:dec_fsigma_I} 
A non-trivial $F_\sigma$ subset $A$ of $\RR$ satisfies $(\text{I}_0)$ if and only if it is the union of a sequence $\vec{A}$ of pairwise disjoint closed sets. 
\end{Thm} 
\begin{proof} 
We first assume that $A$ is the union of a sequence $\vec{A}$ of pairwise disjoint closed sets. If $I$ is an open interval that is contained in $A$, then it is contained in one of the sets by Fact \ref{subinterval of union of disjoint closed sets} and therefore its end points are in $A$, thus proving $(\text{I}_0)$. 

Assuming conversely that $A$ satisfies $(\text{I}_0)$, we construct a reduction $f$ on $\RR$ of $A$ such that $f(A)$ is disjoint from a countable dense subset of $\RR$. 
To this end, let $\vec{A}=\langle A_n\mid n\in\omega\rangle$ enumerate all non-trivial connected components of $A$ and singletons that come from a countable set that is disjoint from $A$ such that $\bigcup\vec{A}$ is dense in $\RR$ and there is such a singleton between any two elements of $\vec{A}$. 
We have that all connected components of $A$ are closed by the assumption $(\text{I}_0)$. 
We now define an increasing sequence $\vec{f}=\langle f_n\mid n\in\omega\rangle$ of increasing functions. 
To this end, let $\langle q_n\mid n\in\omega\rangle$ enumerate $\QQ$ without repetitions 
and begin by letting $f_0=c_{q_0}{\upharpoonright}A_0$. 
In the following, we will write $f_n(\infty)=\infty$ and $f_n(-\infty)=-\infty$ for all $n\in\omega$ for notational convenience. 
To define $f_n$ for $n=m+1$, let $x$ be the largest element of $\dom(f_m)$ below $A_n$ if it exist and $-\infty$ otherwise and further $y$ the least element of $\dom(f_m)$ above $A_n$ of this exists and $\infty$ otherwise. 
Now let $k_n$ be least with $q_{k_n}\in (f_m(x),f_m(y))$ and extend $f_m$ to $f_n$ by mapping $A_n$ to $q_{k_n}$. 
Finally, let $f$ denote the union of the sequence $\vec{f}$. 
It is clear from the construction that $f$ is an increasing function with a unique continuous extension $f^\star$ to $\RR$ 
and we further claim that $f^\star$ is a reduction of $A$ to $f^\star(A)$. This follows immediately from the fact that $f^\star$ is strictly increasing with respect to elements of different sets in $\vec{A}$ by the construction. 
It is also clear that $f^\star(A)$ is disjoint from a countable dense subset of $\QQ$ 
and thus the complement $B$ of this set is homeomorphic to the Baire space. 
Assuming that $\vec{C}=\langle C_n\mid n\in\omega\rangle$ is a sequence of compact sets with union $A$, there is a sequence $\vec{D}=\langle D_n\mid n\in\omega\rangle$ of pairwise disjoint relatively closed subsets of $B$ that refines $\langle f^\star(C_n)\mid n\in\omega\rangle$ by Fact \ref{lem:Baire space F_sigma}. 
We thus obtain the required decomposition by pulling back these sets via $f^\star$. 
\end{proof}

The idea of the previous construction is taken up again in the proofs of Theorem \ref{lem:red_Q_I} and Lemmas \ref{lem:red_closed} and \ref{prop:red_fsigma} below in more complex settings.



We now aim to characterize the sets that are reducible to $\QQ$ by the condition $(\text{I})$. We first recall that by Lemma \ref{I}, for any non-trivial set $A$ the condition $(\text{I})$ is equivalent to the statement that no non-trivial open or closed sets reduce to $A$. 
Moreover, the next result shows that $(\text{I})$ is not easy to satisfy.

\begin{Lem} \label{non delta02} 
The condition $(\text{I})$ fails for every non-trivial $\mathbf{\Delta}^0_2$ subset $A$ of $\RR$. 
\end{Lem}

\begin{proof}
Towards a contradiction, assume that $(\text{I})$ holds for $A$. Since the sets $\partial A \setminus A$ and $\partial A\cap A$ are disjoint $\mathbf{\Delta}^0_2$ subset of $\partial A$, by the Baire category theorem for $\partial A$ one of them is not dense in $\partial A$. \footnote{We would like to thank the referee for simplifying this step of the proof.}  
Assuming that $\partial A\setminus A$ is not dense in $\partial A$, there is an open interval $I$ with $\partial A\cap I$ nonempty and contained in $A$. 
It follows that $A\cap I$ is closed in $I$, since any limit of $A\cap I$ in $I\setminus A$ would be an element of $\partial A\cap I\subseteq A$. 
Moreover, both $A\cap I$ and $I\setminus A$ are nonempty and hence there is a maximal subinterval of $I$ that is disjoint from $A$ with one of its end points in $A$. 
However, this shows that $A$ does not satisfy $(\text{I}_1)$ and thus contradicts our assumption. 
Since $(\text{I})$ also holds for the complement of $A$ by our assumption, we also obtain a contradiction in the case that $\partial A\cap A$ is dense in $A$. 
\end{proof}

\begin{Thm} \label{lem:red_Q_I} 
A non-trivial $F_\sigma$ subset of $\RR$ satisfies $(\text{I})$ if and only if it is reducible to $\QQ$. 
\end{Thm}
\begin{proof} 
The first implication follows from the fact that $(\text{I})$ is preserved under forming continuous preimages.

For the other implication, suppose that $A$ is a non-trivial $F_\sigma$ subset of $\RR$ that satisfies (I). 
We will construct an increasing sequence of partial functions on $\RR$ whose union has a unique continuous extension to $\RR$ that is a reduction of $A$ to $\QQ$. 
To this end, suppose that $\vec{p}=\langle p_i\mid i\in\omega\rangle$ and $\vec{q}=\langle q_i\mid i\in\omega\rangle$ enumerate $\QQ$ and a dense subset of $\mathbb{R} \setminus \mathbb{Q}$. 
By Theorem \ref{prop:dec_fsigma_I}, there is a partition $\vec{A}=\langle A_n \mid n\in \omega\rangle $ of $A$ into closed sets; 
we can further assume that each $A_n$ is either compact or unbounded and connected by subdividing these sets. 
Moreover, we pick a sequence $\vec{B}=\langle B_n \mid n\in \omega\rangle$ of distinct connected components of the complement $B$ of $A$ that includes all non-trivial ones; these are closed by condition $(\text{I}_1)$ for $A$. 
We further write $A^n=\bigcup_{m\leq n} A_m$ and $B^n=\bigcup_{m\leq n} B_m$. 
We now claim that we can assume that the convex closure of $A_{n+1}$ is disjoint from $A^n\cup B^n$ for all $n\in\omega$. 
This follows from the fact that we can partition each $A_{n+1}$ into finitely many components whose convex closures are disjoint from $A^n\cup B^n$, since we would otherwise obtain an element of $A_{n+1}\cap(A^n\cup B^n)$ by compactness.

We now construct an increasing sequence $\vec{f}=\langle f_n\mid n\in\omega\rangle$ of functions and a sequence $\vec{I}=\langle I_n\mid n\in\omega\rangle$ of open intervals with the following properties: 
the domain of each $f_n$ is the union of the elements of a finite set $\Gamma_n$ of sets of the form $A_i$ and $B_j$ that contains $A^n\cup B^n$ and the range of each $f_n$ is finite; we have $I_n\cap \{p_0,\dots,p_n\}=\emptyset$ and $\cl(I_k)\subseteq I_n$  for all $k\in\omega$ with $\inf_{I_k}\in I_n$. 
Moreover, we will write $f_n(\infty)=\infty$ and $f_n(-\infty)=-\infty$ for all $n\in\omega$ 
for notational convenience. 
We begin the construction by letting $f_0{\upharpoonright}A_0=c_{p_0}{\upharpoonright}A_0$ and $f_0{\upharpoonright}B_0=c_{q_0}{\upharpoonright}B_0$. 
We further let $\vec{I}$ initially denote the empty sequence and will extend it in each step. 
Assuming that $n=m+1$, we first let $f_n=f_m$, $\Gamma_n=\Gamma_m$ and extend 
them 
by applying the following step successively to finitely many closed sets $C$.

\setcounter{Step}{0}

\begin{Step*} 
Let $x$ be the largest element of $\dom(f_n)$ below $\min_C$ if this exists and $-\infty$ otherwise; 
further let $y$ be the least element of $\dom(f_n)$ above $\max_C$ if this exists and $\infty$ otherwise. 
Let $\vec{r}=\langle r_k\mid k\in\omega\rangle$ be equal to $\vec{p}$ if $C$ is a subset of $A$ and  equal to $\vec{q}$ otherwise. 
\begin{itemize} 
\item[$\sqbullet$]
If $f_n(x)\neq f_n(y)$, then we pick the least $k\in\omega$ with $r_k\in (f_n(x),f_n(y))^\star$ and let $f_n{\upharpoonright}C=c_{r_k}{\upharpoonright}C$ 
\item[$\sqbullet$]
If $f_n(x)=f_n(y)$, 
then we choose an open interval $I$ with minimum $f_n(x)$ that has length at most $\frac{1}{n}$, is disjoint from 
$\{q_0,\dots, q_i\}$ for the current length $i$ of $\vec{I}$ and 
whose closure is contained in any element of $\vec{I}$ that contains $f_n(x)$; 
then pick the least $k\in\omega$ with $r_k\in I$ and let $f_n{\upharpoonright}C=c_{r_k}{\upharpoonright}C$ 
\end{itemize} 
\end{Step*}

We first perform this step successively for each set $A_n$ and $B_n$ whose value is not yet defined. We then consider each maximal open interval 
between elements of $\Gamma_m\cup\{A_n,B_n\}$ separately; for each choose the least $k$ with $A_k$ contained in it and apply the previous step. Moreover, we update $\Gamma_n$ and in each step $\vec{I}$ by adding the new sets and intervals. 
Finally, $f$ is defined as the union of $\vec{f}$ and $s_n$ and $t_n$ denote the values of $A_n$ and $B_n$ for all $n\in\omega$.

Before we prove that $f$ has the required properties, we introduce the following terminology. 
We begin by fixing $\vec{m}=\langle m_i\mid i\in\omega\rangle$ and $\vec{n}=\langle n_i\mid i\in\omega\rangle$ as notations for strictly increasing sequences and will write $A(\vec{m})$ 
for the sequence $\langle A_{m_i}\mid i\in\omega\rangle$. 
\begin{itemize} 
\item[$\sqbullet$]
The \emph{critical number} of a set or a real number is the largest $i\in\omega$ such that this is contained in $\conv (A_i)$ 
and $-1$ if there is no such number 
\item[$\sqbullet$]
A set $A_i$ is \emph{fresh at $x$} if no real number strictly between $A_i$ and $x$ is in the domain at the stage where the value $s_i$ of $A_i$ is defined 
\item[$\sqbullet$]
A sequence $A(\vec{n})$ is \emph{fresh at $x$} if is strictly monotone, converges to $x$ and consists solely of fresh sets at $x$ with the same critical number that are listed in the same order as they appear in the construction 
\item[$\sqbullet$]
A sequence $A(\vec{n})$ is \emph{nested at $x$} if $A_{n_{i+1}}\subseteq \conv(A_{n_i})$ and $x\in\conv(A_{n_i})$ for all $i\in\omega$ 
\end{itemize} 
We will further use the following facts that follow immediately from the construction above. 
\begin{itemize} 
\item[$\sqbullet$]
No set $A_i$ or $B_j$ is contained in the convex closure of a set that occurs later in the construction 
\item[$\sqbullet$]
If $f$ maps some set to an interval $I$ in $\vec{I}$, then it maps its convex closure to $I$ as well 
\item[$\sqbullet$]
$f$ is strictly monotone with respect to elements of distinct sets of the form $A_i$ or $B_j$ with the same critical number 
\item[$\sqbullet$]
If a sequence $A(\vec{n})$ or $B(\vec{n})$ converges to an element of $A\cup B$ and has constant critical number, then $f$ is strictly monotone with respect to elements of distinct sets in this sequence and its limit point 
\item[$\sqbullet$]
If there is no nested sequence at some $x\in\partial A$, then there is a fresh sequence $A(\vec{n})$ at $x$; if additionally $x\in A \cup B$, then we can choose $A(\vec{n})$ on any side of $x$ where it is approachable 
\end{itemize}

\begin{Claim*} 
$f$ has a unique continuous extension $f^\star$ to $\RR$. 
\end{Claim*} 
\begin{proof} 
It is sufficient to show that for any $x\in\partial A$, there is a unique value that is equal to $\lim_{i\rightarrow\infty} s_{m_i}=\lim_{i\rightarrow\infty} t_{n_i}$ for all sequences $A(\vec{m})$ and $B(\vec{n})$ that converge to $x$. 
If there is a nested sequence $A(\vec{m})$ at $x$, then $f$ maps the sets $\conv(A_{m_i})$ to a nested sequence of intervals whose lengths converge to $0$ 
and hence all other sequences that are considered in the statement converge to $\lim_{i\rightarrow\infty}s_{m_i}$ as well. 
We can thus assume that this condition fails.

We first assume that $x\in A\cup B$ and will prove that $f(x)$ is the required value. To this end, it suffices to take $A(\vec{m})$ converging from below, since a similar proof 
works for the remaining cases. 
We can further assume that the critical number of $A(\vec{m})$ is constant by passing to a tail of $\vec{m}$. 
Since there is no nested sequence at $x$, there is a 
fresh sequence 
$A(\vec{n})$ 
at $x$ 
that converges from below. 
Since $f$ is strictly monotone with respect to each of the sequences $A(\vec{m})$ and $A(\vec{n})$, 
we can further assume that 
the sets in the sequences 
$A(\vec{m})$ and $A(\vec{n})$ are pairwise disjoint 
by thinning out $\vec{m}$ and $\vec{n}$. 
Then $f$ is strictly monotone simultaneously with respect to both sequences  
and their limit $x$; 
we can assume that it is strictly increasing and thus 
$\lim_{i\rightarrow\infty} s_{m_i}=\lim_{i\rightarrow\infty} s_{n_i}\leq f(x)$. 
Towards a contradiction, suppose that this inequality is strict and pick the least $j\in\omega$ with $\lim_{i\rightarrow\infty} s_{n_i}\leq p_j< f(x)$. 
Let $p_{k_0},\dots, p_{k_j}$ denote the images of $A_{n_0}, \dots,  A_{n_j}$; these values are strictly increasing and since these sets are fresh at $x$, the values $k_0,\dots,k_j$ are also strictly increasing. 
Then $s_{k_i}$ 
is necessarily chosen as $p_j$ in the construction for the least $i\leq j$ with $k_i\geq j$, but this contradicts the fact that $f$ is strictly increasing on $A(\vec{n})$. 

We now assume that $x\notin A\cup B$. 
Since the limits in the statement are equal for sequences with the same direction, it is sufficient to prove that $\lim_{i\rightarrow\infty} s_{m_i}=\lim_{i\rightarrow\infty} s_{n_i}$ for sequences $A(\vec{m})$ converging from below and $A(\vec{n})$ from above; the proof is similar for the remaining cases. 
Since there is no nested sequence at $x$, there is a fresh sequence at $x$; we can assume that it converges from below and moreover, since the required limit is equal to that of the values of $A(\vec{m})$, we can assume that $A(\vec{m})$ is already fresh. 
Since we can further assume that 
the sets in $A(\vec{m})$ and $A(\vec{n})$ are pairwise disjoint and have the same critical number by thinning out $\vec{m}$ and $\vec{n}$, we have that $f$ is strictly monotone with respect to distinct elements of $A(\vec{m})$ and $A(\vec{n})$; we can assume that it is strictly increasing and $\lim_{i\rightarrow\infty} s_{m_i}\leq \lim_{i\rightarrow\infty} s_{n_i}$. 
Towards a contradiction, suppose that this inequality is strict and 
pick the least $j$ with $\lim_{i\rightarrow\infty} s_{m_i}\leq p_j \leq\lim_{i\rightarrow\infty} s_{n_i}$. 
Let $p_{k_0},\dots, p_{k_j}$ denote the images of $A_{m_0}, \dots,  A_{m_j}$; these values are strictly increasing and since these sets are fresh at $x$, the values $k_0,\dots,k_j$ are also strictly increasing. 
Hence $s_{k_i}$ is chosen as $p_j$ for the least $i\leq j$ with $k_i\geq j$, contradicting the assumptions. 
\end{proof}

It remains to show that $f^\star$ is a reduction of $A$ to $\QQ$. 
The construction guarantees that $f(A)$ is a subset of $\QQ$ and $f(B)$ is disjoint from $\QQ$. 
Towards a contradiction, suppose that there is some $x\notin C$ and some $j\in\omega$ with $f^\star(x)=p_j$. 
If there is a nested sequence $A(\vec{n})$ at $x$, then $f$ maps $\conv(A_{n_i})$ to an interval $I$ whose closure is disjoint from $\{p_0,\dots,p_j\}$ if $i\in\omega$ is  sufficiently large, 
but this contradicts the choice of $p_j$. 
Assuming that there is no nested sequence, 
we pick a fresh sequence $A(\vec{n})$ and moreover assume that it converges from below and that $f$ is strictly increasing with respect to sets in $A(\vec{n})$.
Now 
let $p_{k_0},\dots, p_{k_j}$ denote the images of $A_{n_0}, \dots,  A_{n_j}$; since these sets are fresh at $x$, the indices $k_0,\dots,k_j$ are also strictly increasing. 
Therefore $s_{k_i}$ is defined as $p_j$ in the construction for the least $i\leq j$ with $k_i\geq j$, contradicting the assumptions. 
\end{proof}

By Lemma \ref{I}, we obtain as a further equivalence to the conditions in Theorem \ref{lem:red_Q_I} that no non-trivial closed or open set is reducible to the $F_\sigma$ set that is considered there. 
Moreover, we can use Theorem \ref{lem:red_Q_I} as follows to classify the countable sets that satisfy $(\text{I}_1)$. First, any such set $A$ is reducible to $\QQ$ and thus clearly either Wadge equivalent to $\QQ$ or nowhere dense. 
Moreover, if $A$ is nowhere dense then $\partial A\cap A=\emptyset$. 
It can be easily shown that for any two such sets, there is a homeomorphism of $\RR$ that maps one of these sets to the other set by enumerating the elements and the maximal closed subintervals of the complement and performing a back-and-forth construction. 
In particular, these sets are Wadge equivalent.

\section{A minimal set below the rationals}\label{section minimal set}

In this section, we construct a minimal subset of $\RR$ below $\QQ$ and 
prove that it is uniquely determined up to Wadge equivalence. 
To this end, let $A$ denote a nonempty compact set with the following properties. 
\begin{itemize} 
\item[$\sqbullet$]
Any connected component 
is approachable (\emph{approachability}) 
\item[$\sqbullet$]
Any end point 
is contained in a subinterval (\emph{thick ends})
\item[$\sqbullet$]
There is a subinterval 
between any two distinct connected components 
(\emph{density of subintervals}) 
\end{itemize} 
In order to see that such a set exists, let $A$ denote the set of end points of the Cantor subset $A_c$ of $[0,1]$, $B$ a countable dense subset of $A_c\setminus A$ and $f$ an order isomorphism from $(0,1)$ onto an open interval such that $A\cup B$ is contained in $f(\QQ_2)$. It is then clear that $(f_c\circ f)^{-1}(A_c)$ has the stated properties. 
Moreover, it is easy to show by a back-and-forth construction that any two such sets are homeomorphic and in particular, any such set is homeomorphic to each of its components. 
For the next two results, we fix such a set 
with minimum $0$ and maximum $1$.


\begin{Lem} \label{lem:red_closed} 
If $B$ is a nonempty compact subset of $(0,1)$, then there is a function $f\colon [0,1]\rightarrow (0,1)$ with $f(\min_A)=\min_B$ and $f(\max_A)=\max_B$ that reduces $A$ to $B$ and is increasing on $A$. 
\end{Lem} 
\begin{proof} 
Since the claim is easy to see if $B$ is a finite union of intervals, we will assume otherwise. 
Let $\vec{A}=\langle A_n\mid n\in \omega\rangle$ be a list of all closed subintervals of $A$ with both ends approachable and $\vec{B}=\langle B_n\mid n\in\omega\rangle$ a list of pairwise disjoint closed connected subsets of $B$ that includes all maximal subintervals and is chosen such that $\bigcup \vec{B}$ dense in $B$. 
Moreover, let $\vec{C}$ and $\vec{D}$ enumerate all maximal open subintervals of the complements of $A$ and $B$, respectively.

We now construct an increasing sequence $\vec{f}=\langle f_n\mid n\in\omega\rangle$ of functions; we will show that its union has a unique continuous extension to the unit interval with the required properties. 
We begin by choosing 
an increasing function $f_0$ that maps the connected components of $\min_A$ and $\max_A$ in $A$ onto the connected components of $\min_B$ and $\max_B$ in $B$. 
Assuming that $n=m+1$, we first let $f_n=f_m$ and then extend $f_n$ by applying finitely many of the following steps in an order that is described below to maximal open subintervals $I=(x,y)$ of the unit interval that are disjoint from $\dom(f_n)$ and to the respective intervals $I_f=(f_n(x),f_n(y))$. 

\setcounter{Step}{0} 

\begin{Step} \label{closed step 1} 
If $I_f\cap B=\emptyset$, then we map $I\cap A$ to $f_n(x)$ and 
choose a continuous extension 
that sends $I\setminus A$ to $I_f$. 
\end{Step}

\begin{Step-2A} \label{closed step 2} 
If $i,j \in\omega$ are least with $A_i$ contained in $I$ and $B_j$ contained in $I_f$, then we map $A_i$ onto $B_j$ by an increasing function. 
\end{Step-2A}

\begin{Step-2B} \label{closed step 3} 
If $i,j \in\omega$ are least with $C_i$ contained in $I$ and $D_j$ contained in $I_f$ and moreover 
$X, Y$ and $X',Y'$ denote the maximal closed subintervals of $A$ and $B$ that are located directly below and above them, 
we proceed as follows. 
\begin{itemize} 
\item[$\sqbullet$]
$C_i$ is mapped onto $D_j$ by an increasing function 
\item[$\sqbullet$]
If $\min_{D_j}= f_n(x)$, then we send $X$ to this value and otherwise map it onto $X'$ by an increasing function; 
we then proceed in the same fashion for $\max_{D_j}$, $f_n(y)$, $Y$ and $Y'$ 
\end{itemize} 
\end{Step-2B} 
We first apply Step 1 to all intervals where it applies and 2A to all other intervals. After updating $f_n$, 
we then apply Steps 1 and 2B in the same fashion to obtain the final version of $f_n$. 
We finally define the function $f$ as $\bigcup \vec{f}$. 

We now show that this function has a unique continuous extension to the unit interval with the required properties. To this end, we 
write 
$\vec{n}=\langle n_i\mid i\in\omega\rangle$ for strictly increasing sequences and $A(\vec{n})$ for the sequence $\langle A_{n_i}\mid i\in\omega\rangle$ 
as in proof of Theorem \ref{lem:red_Q_I}; we further fix the following notation. 
\begin{itemize} 
\item[$\sqbullet$]
A set $A_i$ is \emph{fresh at $x$} if no real number strictly between $A_i$ and $x$ is in the domain at the stage where $f{\upharpoonright}A_i$ is defined 
\item[$\sqbullet$]
A sequence $A(\vec{n})$ is \emph{fresh at $x$} if is strictly monotone, converges to $x$ and consists only of fresh sets at $x$ that are listed in the same order as they appear in the construction 
\end{itemize} 
Moreover, we will use the following facts which result immediately from the above construction. 
\begin{itemize} 
\item[$\sqbullet$]
$\dom(f)$ is the union of $\bigcup\vec{A}$ with the subintervals of $A$ that contain end points of $A$ and its complement $\bigcup\vec{C}$; 
in particular it is dense in the unit interval
\item[$\sqbullet$]
$\ran(f)$ is the union of subsequences of $\vec{B}$ and $\vec{D}$ 
\item[$\sqbullet$]
$f$ 
is a partial reduction of $A$ to $B$ which maps all sets in $\vec{A}$ and $\vec{C}$ to sets in $\vec{B}$ and $\vec{D}$, respectively 
\item[$\sqbullet$]
$f{\upharpoonright}A$ is increasing;  
if moreover two connected components of $A$ 
have distinct images, then $f$ maps the open interval between them strictly between their images  
\item[$\sqbullet$]
Images of sequences with the same limit and direction converge to the same value; this follows from the previous item 
\item[$\sqbullet$]
If $x\in\dom(f)$ is approachable on one side with respect to $A$, then there is a fresh sequence at $x$ on the same side 
\item[$\sqbullet$]
If $x\notin \dom(f)$ is approachable from both sides with respect to $A$, then there is a fresh sequence at $x$ 
\end{itemize}

To show that $f$ is continuous, it is sufficient by the last fact to prove 
that $\lim_{i\rightarrow\infty} f(x_i)=f(x)$ for every strictly monotone sequence $\vec{x}=\langle x_i\mid i\in\omega\rangle$ in $A$ that converges to some $x\in\dom(f)$. We can further assume that both $\vec{x}$ and the sequence of its images are strictly increasing and $x$ is approachable from below, since the case in which $\vec{x}$ is decreasing is symmetric and the remaining cases are easy.  
We can thus pick a fresh sequence $A(\vec{n})$ that converges to $x$ from below. 
Now let $\vec{k}=\langle k_i\mid i\in\omega\rangle$ such that $A_{n_i}$ is mapped to $B_{k_i}$ for all $i\in\omega$; this sequence is strictly increasing since each $A_{n_i}$ is fresh at $x$. 
Since $f$ is increasing on $A$, the sets $B_{k_i}$ moreover converge to $\lim_{i\rightarrow\infty} f(x_i)$ 
and we now assume towards a contradiction 
that this value is strictly below $f(x)$. 
Since $B$ is closed, there is some $j\in\omega$ such that $B_j$ is located strictly between all 
$B_{k_i}$ and $f(x)$. We now let $j\in\omega$ be the least such number and let further $i\in\omega$ be least with $k_i\geq j$. Then the image $B_{k_i}$ of $A_{n_i}$ is necessarily chosen as $B_j$ in the construction, but this contradicts the choice of $B_j$ above the limit.

To prove that $f$ has a unique continuous extension $f^\star$ to the unit interval, it remains to show that $\lim_{i\rightarrow\infty} f(x_i)=\lim_{i\rightarrow\infty} f(y_i)$ for all strictly increasing and strictly decreasing sequences $\vec{x}=\langle x_i\mid i\in\omega\rangle$ and $\vec{y}=\langle y_i\mid i\in\omega\rangle$ in $A$ 
that converge to the same $x\notin \dom(f)$. 
We can further assume that the 
images of $\vec{x}$ and $\vec{y}$ are strictly increasing and decreasing and $x$ is approachable from both sides, since otherwise the claim 
follows from continuity of $f$. 
Let $A(\vec{n})$ be fresh at $x$; we can further assume that it is increasing, since the other case is symmetric. 
Now let $\vec{k}=\langle k_i\mid i\in\omega\rangle$ such that $A_{n_i}$ is mapped to $B_{k_i}$ for all $i\in\omega$; this sequence is strictly increasing since each $A_{n_i}$ is fresh at $x$. 
Since $f$ is increasing on $A$, the sets $B_{k_i}$ moreover converge to $\lim_{i\rightarrow\infty} f(x_i)$ 
and we now assume towards a contradiction 
that this value is strictly below $\lim_{i\rightarrow \infty}f(y_i)$. 
Since $B$ is closed, there is some $j\in\omega$ such that $B_j$ is located strictly between all $f(x_i)$ and $f(y_i)$; we fix the least such 
and let further $i\in\omega$ be least with $k_i\geq j$. Then the image $B_{k_i}$ of $A_{n_i}$ has to be chosen as $B_j$, contradicting the choice of $B_j$ above the limit.

It remains to show that $f^\star$ is a reduction of $A$ to $B$ that is increasing on $A$. 
Since this statement holds for $f$ by the construction and $B$ is closed, it also holds for $f^\star$ as required. 
\end{proof}


We now fix a subset $A$ of the unit interval containing $0$ and $1$ that is the union of a sequence $\vec{A}=\langle A_i\mid i\in\omega\rangle$ of disjoint sets, each of whom is homeomorphic to the set that is considered in Lemma \ref{lem:red_closed}, with the following properties for all $i,j\in\omega$ with $i\leq j$. 


\begin{itemize} 
\item[$\sqbullet$]
Any open interval disjoint from $A_i$ with an end point in $A_i$ contains some set of the form $A_k$ and some interval that is disjoint from $A$ (\emph{approachability})
\item[$\sqbullet$]
The convex closure of $A_j$ 
is either disjoint from or contained in that of $A_i$ 
(\emph{no overlaps}) 
\end{itemize} 
It can be shown by a back-and-forth argument that any two sets with these properties are homeomorphic. 


In the following result, all sets and functions live in the unit interval and in particular, we consider property (I) for subsets of the unit interval. 

\begin{Lem} \label{prop:red_fsigma} 
$A$ is reducible to any $F_\sigma$ subset $B$ of the unit interval that satisfies (I) and contains $0$ and $1$ by a function that fixes both $0$ and $1$. 
\end{Lem}

\begin{proof} 
We can assume that the diameters of the sets in $\vec{A}$ converge to $0$ by splitting these sets into finitely many pieces and moreover that 
$0$ and $1$ are contained in $A_0$ and $A_1$, respectively. 
We further pick a partition $\vec{B}=\langle B_i\mid i\in\omega\rangle$ of $B$ into disjoint closed sets by Lemma~\ref{prop:dec_fsigma_I}. 
We can assume that $\vec{B}$ has no overlaps by splitting up the sets in $\vec{B}$ and that $0$ and $1$ are contained in $B_0$ and $B_1$, respectively. 
Let further $\vec{C}=\langle C_i\mid i\in\omega\rangle$ list all maximal closed intervals that are disjoint from $A$. 
We can assume that the convex closure of $A_j$ is disjoint from $C_i$ 
for all $i\leq j$ by splitting the sets in $\vec{A}$ into finitely many pieces. 
Finally, we fix a sequence  $\vec{D}=\langle D_n\mid n\in\omega\rangle$ of pairwise disjoint closed connected subsets of $B$ that includes all maximal subintervals and is chosen such that $\bigcup \vec{D}$ dense in the complement of $B$ and there is a singleton $D_k$ between any two distinct sets $D_i$ and $D_j$, the latter being possible since $B$ has uncountably many connected components in any open interval that does not contain it by Lemma \ref{non delta02}. 


We will construct increasing sequences $\vec{f}=\langle f_n\mid n\in\omega\rangle$ of functions, $\vec{\Gamma}=\langle\Gamma_n\mid n\in\omega\rangle$ of finite sets and $\vec{\epsilon}=\langle \epsilon_n\mid n\in\omega\rangle$ of rationals 
with the following properties for all $m\leq n$. 
\begin{itemize} 
\item[$\sqbullet$]
$\Gamma_m$ consists of sets in $\vec{A}$ and $\vec{C}$ 
\item[$\sqbullet$]
Each function $f_m$ is a finite union of monotone functions on all sets in $\Gamma_m$ where 
\begin{itemize} 
\item[$\blacktriangleright$]
each set $A_i$ is mapped onto a component of a set $B_j$ 
\item[$\blacktriangleright$]
each set $C_i$ is mapped onto a set $D_j$ 
\end{itemize} 
\item[$\sqbullet$]
Given any two successive sets in $\Gamma_m$ with disjoint images, 
$f_n$ maps the open interval between them strictly between their images (\emph{coherence})
\item[$\sqbullet$]
$\epsilon_{n+1}\leq \frac{\epsilon_n}{2}$ 
\end{itemize}


We begin the construction by 
defining $f_0$ on $A_0$ and $A_1$ by an application of Lemma \ref{lem:red_closed} to these sets with images $B_0$ and $B_1$ 
and further let $\Gamma_0=\{A_0,A_1\}$ and $k_0=1$. 
In the induction step $n=m+1$, we first let $f_n=f_m$, $\Gamma_n=\Gamma_m$, $k_n=k_m$ and then extend the former two and increase the latter finitely often.  
To this end, we will perform the following steps for all maximal open intervals $I=(x,y)$ that are strictly between two sets in $\Gamma_n$. 

\setcounter{Step}{0} 

\begin{Step} \label{minimal step 1} 
We pick the least $i\in\omega$ with $A_i$ contained in $I$. 
\begin{itemize} 
\item[$\sqbullet$]
If $f_n(x)\neq f_n(y)$, let $K=(f_n(x),f_n(y))^\star$ and pick the least $j\in\omega$ with $B_j\cap K\neq\emptyset$; 
map $A_i$ onto $B_j\cap K$ by an increasing function if $f_n(x)<f_n(y)$ and a decreasing function otherwise by applying Lemma \ref{lem:red_closed} 
\item[$\sqbullet$]
If $f_n(x)=f_n(y)$, 
let $K_\delta$ be intervals of  length $\delta$ such that $f_n(x)$ is an end point on a side from which it is approachable for any $\delta>0$; 
choose $\delta$ 
such that $K_\delta$ is 
disjoint from $B_j$ for all $j<n$ 
and make $\epsilon_n\leq \delta$ sufficiently small such that $K_{\epsilon_n}$ satisfies coherence; 
pick the least $j\in\omega$ with $B_j\cap K_{\epsilon_n}\neq\emptyset$; 
finally split $A_n$ into two homeomorphic relatively convex subsets and map them 
to $B_j\cap K_{\epsilon_n}$ by monotone 
functions by applying Lemma \ref{lem:red_closed}; 
if $K_{\epsilon_n}$ is above $f_n(x)$ then the function is increasing on the lower part and increasing on the upper part and the direction is reversed if $K_{\epsilon_n}$ is below $f_n(x)$ 
\end{itemize} 
\end{Step}

\begin{Step} \label{ minimal step 2} 
We pick the least $i\in\omega$ with $C_i$ contained in $I$. 
\begin{itemize} 
\item[$\sqbullet$]
If $f_n(x)\neq f_n(y)$, let $K=(f_n(x),f_n(y))^\star$ and pick the least $j\in\omega$ with $D_j\cap K\neq\emptyset$; 
map $C_i$ onto $D_j\cap K$ by an increasing function if $f_n(x)<f_n(y)$ and a decreasing function otherwise by applying Lemma \ref{lem:red_closed} 
\item[$\sqbullet$]
If $f_n(x)=f_n(y)$, 
let $K_\delta$ be intervals of length $\delta$ such that $f_n(x)$ is an end point on a side from which it is approachable for any $\delta>0$; 
choose $\delta$ 
such that $K_\delta$ 
is disjoint from $B_j$ for all $j<n$ 
and make $\epsilon_n\leq \delta$ sufficiently small such that $K_{\epsilon_n}$ satisfies coherence; 
pick the least $j\in\omega$ with $B_j\cap K_{\epsilon_n}\neq\emptyset$; 
finally we map $C_i$ to $D_j$ 
\end{itemize} 
\end{Step}

We first apply Step 1 to all intervals and after updating $f_n$, $\Gamma_n$ and $\epsilon_n$, we apply Step 2 to all new intervals. 
The second part of each step works since the assumptions on $f_n$ imply that $\partial A$ is mapped to $\partial B$ and therefore $f_n(x)$ is approachable with respect to $B$. 
We finally define $f$ as the union of $\vec{f}$.

We now show that this function has a unique continuous extension to the unit interval with the required properties. To this end, we 
write 
$\vec{n}=\langle n_i\mid i\in\omega\rangle$ for strictly increasing sequences and $X(\vec{n})$ for the sequence $\langle X_{n_i}\mid i\in\omega\rangle$, 
where $X$ denotes either $A$, $B$, $C$ or $D$. 
To define the following notation, we write $X$ for $A$ or $C$ and $Y$ for sets of the form $A_i$ or $C_i$. 
\begin{itemize} 
\item[$\sqbullet$]
A set $Y$ is \emph{fresh at $x$} if no real number strictly between $Y$ and $x$ is in the domain at the stage where $f{\upharpoonright}Y$ is defined 
\item[$\sqbullet$]
A sequence $X(\vec{n})$ is \emph{fresh at $x$} if is strictly monotone, converges to $x$ and consists only of fresh sets at $x$ that are listed in the same order as they appear in the construction 
\item[$\sqbullet$]
A sequence $X(\vec{n})$ is \emph{nested at $x$} if $X_{n_{i+1}}\subseteq \conv(X_{n_i})$ and $x\in\conv(X_{n_i})$ for all $i\in\omega$ 
\end{itemize} 
Moreover, we will use the following statements which follow immediately from the construction. 
\begin{itemize} 
\item[$\sqbullet$]
Assuming that there is a nested sequence at $x$, the images of all sequences in $\dom(f)$ that converge to $x$ converge to the same limit 
\item[$\sqbullet$]
Assuming that there is no nested sequence at $x$ 
\begin{itemize} 
\item[$\blacktriangleright$]
if $x\in\dom(f)$ is approachable on one side with respect to $A$, then there are fresh sequences $A(\vec{m})$ and $C(\vec{n})$ at $x$ on the same side 
\item[$\blacktriangleright$]
if $x\notin \dom(f)$ is approachable from both sides with respect to $A$, then there 
are fresh sequences $A(\vec{m})$ and $C(\vec{n})$ at different sides of $x$ 
\end{itemize} 
\end{itemize}

In the next statement $X$, $Y$ and $Z$ stand for sets of the form $A_i$ or $B_i$. 
For any $x\notin \dom(f)$ such that there is no nested sequence at $x$, we fix fresh sets $X$ below $x$ and $Y$ above $x$ such that there is no set $Z$ between them with $x\in\conv(Z)$. We say that $f$ is \emph{increasing at $x$} if the image of $X$ is strictly above the image of $Y$ and \emph{decreasing at $x$} if the order is reversed. 

For the next two claims, we assume that $x\notin \dom(f)$ and there is no nested sequence at $x$. 

\begin{Claim*} 
The limits of the images of all sequences that converge to $x$ from one side take fixed values $x_{\mathrm{low}}$ for sequences below and $x_{\mathrm{high}}$ for sequences above with the following properties. 
\end{Claim*} 
\begin{proof} 
We can assume that $f$ is increasing at $x$, since the proof of the other case is symmetric. 
For the first claim, it is sufficient to show that the image of any sequence $\vec{x}=\langle x_i\mid i\in\omega\rangle$ that converges to $x$ from below is Cauchy. 
To this end, we construct a fresh sequence that converges to $x$ from below as follows. 
Suppose that $X$ and $Y$ witness that $f$ is increasing at $x$. 
We now let $X_i$ denote sets of the form $A_j$ or $B_j$. 
We pick $X_0$ as the first set in the construction after $X$ and $Y$ that is strictly between them and fresh below $x$ with $x_0\leq y$ for some $y\in X_0$. 
Given $X_i$, we pick $X_{i+1}$ as the first set in the construction appearing after and above $X_i$ with $x_{i+1}\leq y$ for some $y\in X_{i+1}$. 
Then $\vec{X}=\langle X_i\mid i\in\omega\rangle$ is fresh. 
Moreover, it is immediate from the construction that the image of $\vec{X}$ is strictly increasing and converges to a real number $x_{\mathrm{low}}$. 
By coherence, $\lim_{i\rightarrow\infty} f(x_i)$ exists and is equal to $x_{\mathrm{low}}$. 
\end{proof} 

Given the previous claim, the following facts are immediate from the construction. 

\begin{itemize} 
\item[$\sqbullet$]
If $f$ is increasing at $x$, then images of fresh sequences that converge to $x$ from one side converge to their limit from the same side and $x_{\mathrm{low}}\leq x_{\mathrm{high}}$ 
\item[$\sqbullet$]
If $f$ is decreasing at $x$, then images of fresh sequences that converge to $x$ from one side converge to their limit from the other side and $x_{\mathrm{high}}\leq x_{\mathrm{low}}$ 
\end{itemize}

In the statement of the next claim, let $X$ stand for components of sets of the form $B_i$ or for sets of the form $D_i$. 
We further say that a set is located \emph{weakly between} two real numbers $x$ and $y$ if it is contained in the closed interval with these end points. 

\begin{Claim*} 
If each of the sequences $\vec{x}=\langle x_i\mid i\in\omega\rangle$ and $\vec{y}=\langle y_i\mid i\in\omega\rangle$ converges to $x$ from one side or is eventually constant with value $x$, 
then no set $X$ has elements that are located weakly between $\lim_{i\rightarrow\infty}f(x_i)$ and $\lim_{i\rightarrow \infty} f(y_i)$. 
\end{Claim*} 
\begin{proof} 
We first assume that $\vec{x}$ converges from below and $\vec{y}$ from above. 
We pick fresh sequences $A(\vec{m})$ below and $C(\vec{n})$ above $x$; 
the proof of the other case is similar. 
Now let $\vec{k}=\langle k_i\mid i\in\omega\rangle$ and $\vec{l}=\langle l_i\mid i\in\omega\rangle$ such that $A_{m_i}$ and $C_{n_i}$ are mapped to $B_{k_i}$ and $D_{l_i}$ for all $i\in\omega$. 
These sequences are both strictly increasing, 
since the sets $A_{m_i}$ and $C_{n_i}$ are fresh at $x$. 
Moreover, the sequences $B(\vec{k})$ and $D(\vec{l})$ 
converge to $\lim_{i\rightarrow\infty} f(x_i)$ and $\lim_{i\rightarrow\infty}f(y_i)$ by the previous claim. 

We now consider two cases for $X$. If $X$ is a component of $B_j$ and its index is least with the stated property, then 
the image $B_{k_i}$ of $A_{m_i}$ is chosen as a component of $B_j$ that contains $X$ for the least $i\in\omega$ with $k_i\geq j$. 
If on the other hand $X$ is equal to $D_j$ and its index is least with the stated property, then 
the image $D_{l_i}$ of $C_{n_i}$ is chosen as $D_j$ for the 
least $i\in\omega$ with $l_i\geq j$. 
However, both conclusions contradict the assumption on $X$. 

We now assume that $\vec{x}$ converges from below and $\vec{y}$ is eventually constant with value $x$. 
We pick fresh sequences $A(\vec{m})$ and $C(\vec{n})$ below $x$. 
Now let $\vec{k}=\langle k_i\mid i\in\omega\rangle$ and $\vec{l}=\langle l_i\mid i\in\omega\rangle$ such that $A_{m_i}$ and $C_{n_i}$ are mapped to $B_{k_i}$ and $D_{l_i}$ for all $i\in\omega$. 
These sequences are both strictly increasing, 
since the sets $A_{m_i}$ and $C_{n_i}$ are fresh at $x$. 
Moreover, the sequences $B(\vec{k})$ and $D(\vec{l})$ 
both converge to $\lim_{i\rightarrow\infty} f(x_i)$ from below by the previous claim. 

We now consider two cases for $X$. If $X$ is a component of $B_j$ and its index is least with the stated property, then 
the image $B_{k_i}$ of $A_{m_i}$ is chosen as a component of $B_j$ that contains $X$ for the least $i\in\omega$ with $k_i\geq j$. 
If on the other hand $X$ is equal to $D_j$ and its index is least with the stated property, then 
the image $D_{l_i}$ of $C_{n_i}$ is chosen as $D_j$ for the 
least $i\in\omega$ with $l_i\geq j$. 
However, both conclusions contradict the assumption on $X$. 

Finally, the remaining cases are symmetric to the previous ones. 
\end{proof}

Since the domain of $f$ is dense in the unit interval, it follows from the previous claim that 
there is a unique continuous extension $f^\star$ to the unit interval. 
It remains to show that $f^\star$ reduces $A$ to $B$. Since $f$ is a partial reduction of $A$ to $B$ whose  domain contains $A$, it remains to show that no $x\notin \dom(f)$ is mapped to $B$. 
If there is a nested sequence at $x$, then this follows from the choice of intervals $K_{\epsilon_n}$ disjoint from $B_i$ for all $i<n$ in the construction. 
If on the other hand there is no nested  sequence at $x$, then the claim follows immediately from the previous claim for components of sets of the form $B_j$. 
\end{proof}

\begin{Thm} \label{minimal set} 
$A$ is reducible to any non-trivial $F_\sigma$ subset of $\RR$ that satisfies (I). 
\end{Thm} 
\begin{proof} 
Assuming that $B$ is such a set, we choose a strictly increasing sequence $\vec{y}=\langle y_n\mid n\in \ZZ\rangle$ in $B$ with no interval $(y_n,y_{n+1})$ contained in $B$. 
We further choose a strictly increasing sequence $\vec{x}=\langle x_n\mid n\in \ZZ\rangle$ in the interior of $A$ with no interval $(x_n,x_{n+1})$ contained in $A$ and no lower or upper bound in $\RR$. 
Then the union of the functions $f_n\colon [x_n,x_{n+1}]\rightarrow [y_n,y_{n+1}]$ given by Lemma \ref{prop:red_fsigma} reduces $A$ to $B$. 
\end{proof}

\section{Gaps and cardinal characteristics} 


The structure of the poset $\mathcal{P}(\omega)/\mathsf{fin}$ has been intensively studied, for instance pertaining 
the existence of gaps and consistent values of cardinal characteristics; for a detailed overview over the theory of gaps in $\mathcal{P}(\omega)/\mathsf{fin}$ and related quasi-orders, we refer the reader to \cite{Scheepers_Gaps} and for cardinals characteristics to \cite{MR2768685}.
Inspired by results in this field, J\"org Brendle asked whether there are gaps in the Wadge quasi-order for the real line. 
This question is answered positively in the next theorem and thereby it is shown that 
the Wadge quasi-order is incomplete. 


\begin{Thm} \label{gaps} 
For all cardinals $\kappa, \lambda\geq 1$ such that there is a $(\kappa,\lambda)$-gap in $\mathcal{P}(\omega)/\mathsf{fin}$, there is a $(\kappa,\lambda)$-gap in the Wadge quasi-order on the Borel subsets of the real line. 
\end{Thm} 
\begin{proof} 
Consider the linear order $L=\ZZ^{(\omega)}$ that was introduced after Lemma \ref{lem:discrete_emb} 
and let $\mapone\colon L^{\otimes 4}\rightarrow\RR$ be a discrete embedding with $\ran(\mapone)$ cofinal and coinitial in $\RR$. 
Moreover, consider the subsets $b^\star$ of $L$ defined after 
and the sets $A_{b^\star}$ defined before Lemma \ref{lem:discrete_emb} 
and further define a function $F$ by letting $F(b)=A_{b^\star}$ for any subset $b$ of $\omega$. 
We further say that an automorphism $\sigma$ of $\ZZ^{(\omega)}$ \emph{fixes $n$} if $\sigma(0^\omega)=\vec{x}=\langle x_i\mid i\in\omega\rangle$ with $x_i=0$ for all $i\geq n$. 
It is easy to see that for any such automorphism $\sigma$ and for any $\vec{y}=\langle y_i\mid i\in\omega\rangle$ and $i\geq n$, we have $\sigma(\vec{y})=\vec{z}=\langle z_i\mid i\in\omega\rangle$ with $y_i=z_i$ for all $i\geq n$. 

We now show that for any gap $(A,B)$ in $\mathcal{P}(\omega)/\mathsf{fin}$ with $A$ and $B$ nonempty, the pair $(F(A),F(B))$ is a gap in the Wadge quasi-order. 
To this end, we first claim that any reduction $f$ of $A_{a^\star}$ to $C$ for any $a\in A$ is surjective. 
To see this, we choose an arbitrary $b\in B$ and a reduction $g$ of $C$ to $A_{b^\star}$. 
By Lemma \ref{images of L-structured sets} the images $f(A_{a^\star})$ and $(g\circ f)(A_{a^\star})$ are $L$-structured by $f\circ \mapone$ and $g\circ f\circ \mapone$, respectively. 
Therefore $\mapone^{-1} \circ(g\circ f\circ\mapone)$ is an automorphism of $L$ and hence $g{\upharpoonright}(f\circ \mapone)(L)\colon (f\circ \mapone)(L)\rightarrow \mapone(L)$ is strictly monotone and bijective. Since $g$ is continuous and $\mapone(L)$ is both cofinal and coinitial in $\RR$, the same holds for $(f\circ \mapone)(L)$ 
and therefore $f$ is surjective. 
To prove the theorem, assume towards a contradiction that some subset $C$ of $\RR$ lies between $F(A)$ and $F(B)$. 
We choose an arbitrary $a_0\in A$ and a reduction $f_0$ of $A_{a_0^\star}$ to $C$. Since $f_0$ is surjective, $C$ is $L$-structured by $f_0 \circ \mapone$ by Lemma \ref{images of L-structured sets} as witnessed by some sequence $\vec{C}=\langle C_i\mid i\in L\rangle$ of compact sets. 
We then define $n\in c$ to hold if $C_{s{}^\smallfrown 1{}^\smallfrown 0^\omega}$ is a singleton for some $s$ of length $n$. 

We now show that $c$ lies between $A$ and $B$ and 
to this end first show that $a\subseteq_{\mathsf{fin}} c$ for all $a\in A$. 
We thus assume that $f$ reduces $A_{a^\star}$ to $C$ 
and that $A_{a_0^\star}$ and  $A_{a^\star}$ are $L$-structured by $\sigma=f_0\circ\mapone$ and $\tau=f\circ \mapone$ as witnessed by $\vec{A}=\langle A_i\mid i\in L\rangle$ and $\vec{B}=\langle B_i\mid i\in L\rangle$. 
Then $\sigma\circ\tau^{-1}$ is an automorphism of $L$ and hence it fixes some $n_0\in\omega$. 
We further have that the set $B_{0^n{}^\smallfrown 1{}^\smallfrown 0^\omega}$ is a singleton for any $n\in a$ by the definition of $A_{a^\star}$. 
If moreover $n\geq n_0$, then $A_{s{}^\smallfrown 1{}^\smallfrown 0^\omega}$ is a singleton for some $s$ of length $n$ and therefore $n\in c$. 
In order to finally show that $c\subseteq_{\mathsf{fin}} b$ for all $b \in B$, we assume that $f$ reduces $C$ to $A_{b^\star}$ and that 
$A_{b^\star}$ is $L$-structured by $\sigma=f\circ f_0\circ \mapone$ and $\mapone$ as witnessed by $\vec{A}=\langle A_i\mid i\in L\rangle$ and $\vec{B}=\langle B_i\mid i\in L\rangle$. 
Then $\sigma\circ\mapone^{-1}$ is an automorphism of $L$ and hence it fixes some $n_0\in\omega$. 
For any $n\in c$, there is some  $s$ of length $n$ such that $C_{s{}^\smallfrown 1{}^\smallfrown 0^\omega}$ is a singleton by the definition of $c$. 
Since $g$ reduces $C$ to $A_{b^\star}$, the set $A_{s{}^\smallfrown 1{}^\smallfrown 0^\omega}$ is a singleton. 
If moreover $n\geq n_0$, then $B_{t{}^\smallfrown 1{}^\smallfrown 0^\omega}$ is a singleton for some $t$ of length $n$. 
Since $\vec{B}$ witnesses that $A_{b^\star}$ is $L$-structured, it follows that $t=0^n$ and $n\in b$. 
\end{proof}

In particular, Theorem \ref{gaps} implies that $(\omega_1,\omega_1)$-gaps, $(\omega,\mathfrak{b})$-gaps and $(\mathfrak{b},\omega)$-gaps exist in the Wadge quasi-order, since it is well-known that \emph{Hausdorff gaps} and \emph{Rothberger gaps} of these types exists in $\mathcal{P}(\omega)/\mathsf{fin}$ (see \cite[Theorem 10, Theorem 19]{Scheepers_Gaps} and  \cite[Theorem 21]{Scheepers_Gaps}, respectively); here $\mathfrak{b}$ denotes the bounding number of the poset $\mathcal{P}(\omega)/\mathsf{fin}$.

Moreover, the similarity between $\mathcal{P}(\omega)/\mathsf{fin}$ and the Wadge quasi-order 
suggests to ask about the values of some cardinal characteristics of the Wadge quasi-order. The next theorem refers to these characteristics and determines the values of the bounding and dominating numbers among others, thus answering questions of Stefan Geschke. 
In the statement of the following fact about definable quasi-orders which we need for its proof, $\mathbb{C}_{\kappa}$ denotes the finite support product of $\kappa$ many copies of Cohen forcing for a cardinal $\kappa$.

\begin{Fact}\label{boost continuum} 
Suppose that $\kappa$ is an infinite cardinal with $\alpha^{\omega}<\kappa$ for all $\alpha<\kappa$. 
Then in any $\mathbb{C}_{\kappa}$-generic extension $V[G]$, 
the length of any quasi-order 
with domain $\mathcal{P}(\omega)$ that is definable from a real and an ordinal is strictly less than $\kappa$. 
\end{Fact}
\begin{proof} 
Suppose that $p\in{\mathbb{C}_{\kappa}}$ forces that $\varphi(x,y,\dot{x},\gamma)$ defines the strict part of a quasi-order on $\mathcal{P}(\omega)$, where $\dot{x}$ is a nice name for a subset of $\omega$. Since can assume that $\dot{x}$ is a name for an element of $V$ by passing to an intermediate extension, we will also write $x<_{\varphi}y$ for $\varphi(x,y,\dot{x},\gamma)$. 
Moreover, suppose that $(\dot{x}_{\alpha}\mid \alpha<\omega_1)$ is a sequence of nice $\mathbb{C}_{\kappa}$-names for subsets of $\omega$ with $p\Vdash_{\mathbb{C}_{\kappa}} \dot{x}_{\alpha}<_{\varphi} \dot{x}_{\beta}$ for all $\alpha<\beta<\omega_1$ -- we can assume that $p=\one_{\mathbb{C}_{\kappa}}$. 
Now let $s_{\alpha}=\text{supp}(\dot{x}_{\alpha}) = \bigcup_{(\check{n} , p ) \in  \dot{x}_{\alpha}} \text{supp} (p)\subseteq \kappa$ for each $\alpha < \kappa$, so each $s_{\alpha}$ is countable. We can assume that $\langle s_{\alpha}\mid \alpha<\omega_1\rangle$ forms a $\Delta$-system by the $\Delta$-system lemma and moreover, 
that the root is empty by passing to an intermediate extension. Thus we assume that $\langle s_{\alpha}\mid \alpha<\omega_1\rangle$ forms a disjoint family. 
Now let $\mathbb{C}_{\kappa}(s)=\{p\in\mathbb{C}_{\kappa}\mid \text{supp}(p)\subseteq s\}$ for $s\subseteq\kappa$. The function that collapses $s_{\alpha}$ to an ordinal induces an isomorphism between $\mathbb{C}_{\kappa}(s_{\alpha})$ and $\mathbb{C}_{\beta_{\alpha}}$ for some $\beta_{\alpha}<\omega_1$ and this function maps $\dot{x}_{\alpha}$ to a nice $\mathbb{C}_{\omega_1}$-name $\pi(\dot{x}_{\alpha})$ for a subset of $\omega$. 
Since there are only $2^{\omega}<\kappa$ many such names and $\kappa$ is regular, there is an unbounded set $I\subseteq\kappa$ with $\pi(\dot{x}_{\alpha})=\pi(\dot{x}_{\beta})$ for all $\alpha,\beta\in I$. 
Finally, assume that $\alpha,\beta\in I$ and $\alpha\neq\beta$. Since $s_{\alpha}$, $s_{\beta}$ are disjoint and $\pi(\dot{x}_{\alpha})=\pi(\dot{x}_{\beta})$, there is an automorphism $\sigma$ of $\mathbb{C}_{\kappa}$ with $\sigma(\dot{x}_{\alpha})=\dot{x}_{\beta}$ and $\sigma(\dot{x}_{\beta})=\dot{x}_{\alpha}$. 
Hence $\one_{\mathbb{C}_{\kappa}}\Vdash \dot{x}_{\alpha}<_{\varphi}\dot{x}_{\beta}$ if and only if $\one_{\mathbb{C}_{\kappa}}\Vdash \dot{x}_{\beta}<_{\varphi}\dot{x}_{\alpha}$, contradicting the assumption that $\varphi$ defines a quasi-order. 
\end{proof}

In the next theorem, the \emph{lengths} $\ell$ and $\ell^\star$ denote the suprema of the sizes of well-ordered and reverse well-ordered subsets, respectively.

\begin{Thm} \label{cardinal characteristics} 
The least size of an unbounded family is $\omega_1$, while the least size of a dominating family and the maximal size of a linearly ordered subset equal $2^\omega$. Moreover, $\ell=\ell^\star$ is consistent with each of the following statements: (a) $\ell=\omega_1$ (b) $\ell=2^\omega$ and (c) $\omega_1<\ell<2^\omega$ and $2^\omega$ is arbitrarily large. 
\end{Thm} 
\begin{proof} 
We fix a sequence $\vec{f}=\langle f_n\mid n\in\omega\rangle$ of homeomorphisms of $\RR$ onto open intervals that are pairwise disjoint. 
Then any sequence $\vec{A}=\langle A_n\mid n\in\omega\rangle$ of subsets of $\RR$ has the upper bound $\bigcup_{n\in\omega} f(A_n)$ in the Wadge quasi-order. 
On the other hand, any uncountable subset of the collection that is given in Theorem \ref{no complete sets} is unbounded 
and 
Theorem \ref{no complete sets} immediately implies that every dominating family has size $2^\omega$. 
To obtain a linearly ordered set of size $2^\omega$ in the Wadge quasi-order, observe that the function $F$ that is defined by $F(x)=\{s\in 2^{<\omega}\mid s{}^\smallfrown 0^\omega\leq_{\mathrm{lex}} x\}$  embeds $(2^\omega,\leq_{\mathrm{lex}})$ into $(\mathcal{P}(2^{<\omega}),\subseteq)$; 
it is further easy to see that the latter is embeddable into $\mathcal{P}(\omega)/\mathsf{fin}$ and this in turn 
into the Wadge quasi-order by Theorem \ref{embed Pomega mod fin}. 

It remains to prove the statement that refers $\ell$ and $\ell^\star$. We first remark that their values are always at least $\omega_1$, since both $\omega_1$ and $\omega_1^\star$ are embeddable into the Wadge quasi-order by Theorem \ref{embed Pomega mod fin}. 
To attain the value $\omega_1$ with large continuum, we force with $\mathbb{C}_\kappa$ for $\kappa=\omega_2$ over a model of the continuum hypothesis $\mathsf{CH}$ and thus obtain the desired values from 
Lemma \ref{boost continuum}. 
Moreover, the consistency of larger values follows from the fact that Martin's Axiom $\mathsf{MA}_{\omega_1}$ for almost disjoint forcing implies that there are well-ordered and reverse well-ordered subsets of $\mathcal{P}(\omega)/\mathsf{fin}$ of size $2^\omega$ (see \cite[Theorem 4.25]{MR1393943}). 
To finally reach values strictly between $\omega_1$ and the continuum, consider the standard model of $\mathsf{MA}_{\omega_1}$ that is obtained by forcing over a model of the generalized continuum hypothesis $\mathsf{GCH}$ and now force with $\mathbb{C}_\kappa$ for $\kappa=\omega_3$. In this extension, the values of both $\ell$ and $\ell^\star$ are equal to $\omega_2$ by Lemma \ref{boost continuum}. 
\end{proof}

\section{
One-sided continuous functions} \label{section one-sided reducibility}

We finally touch the topic of reducibilities with respect to other classes of functions. 
For instance, one obtains a reducibility with a much simpler structure 
by slightly expanding the class of continuous functions. 
Recall that a function $f$ on $\RR$ is called \emph{right-continuous} at $x$ if for every $\epsilon>0$, there is some $\delta>0$ such that for all $y$ with $x<y<x+\delta$, we have $|f(x)-f(y)|<\epsilon$. 
As reductions, we consider the class of finite compositions of right-continuous functions and let $\preceq$ and $\prec$ denote the corresponding notions of reducibility and strict reducibility. 
We will obtain the $\mathsf{SLO}$ principle for this reducibility by comparing it with the Wadge quasi-order for the Baire space via a function as in the following fact.

\begin{Fact} \label{continuous bijection} 
There is a continuous bijection $f\colon {}^{\omega}\omega\rightarrow \RR$ 
with a right-continuous inverse. 
\end{Fact} 
\begin{proof} 
We fix an enumeration $\langle n_i\mid i\in\omega\rangle$ of $\ZZ$. 
It is easy to show by a back-and-forth construction that the linear orders $({}^\omega\omega , \leq_{\text{lex}})$ and $\bigl([0,1) , \leq\bigr)$ are isomorphic, 
since both of them 
are dense, have dense countable subsets and a minimum but no maximum. 
We can thus pick order isomorphisms $f_i\colon {}^\omega\omega\rightarrow [n_i,n_{i+1})$. 
In particular, $f_i$ is a homeomorphism with respect to the order topologies on both spaces. 
Since the product topology on the Baire space is finer than the order topology on $({}^\omega\omega , \leq_{\text{lex}})$, we have that 
$f_i$ is also continuous with respect to the former. 
Moreover, it is easy to see that $f_i^{-1}$ is right-continuous, since the images of a strictly decreasing sequence eventually stabilize in each coordinate. 
In order to glue these maps $f_i$ together, 
we will use the following notation. If $x$ is an element of the Baire space, we write $x^+$ for the element of ${}^\omega\omega$ that is obtained by shifting the coordinates by one and thus define $x^+(i)=x(i+1)$ for all $i\in\omega$. 
We can now define 
$f\colon {}^\omega\omega\rightarrow \RR$ by letting $f(x)=f_{x(0)}(x^+)$. It is then easy to check that $f$ is continuous and its inverse is right-continuous. 
\end{proof}

In the statement of the next theorem, let $\mathcal{C}$ denote any nonempty class of subsets of Polish spaces that is closed under forming complements, countable unions and preimages with respect to Borel-measurable functions such that all subsets of the Baire space in $\mathcal{C}$ are determined, for instance the class of Borel sets.

\begin{Thm}\label{prop:right_cont}
The restriction of $\preceq$ to the class of subsets of $\RR$ in $\mathcal{C}$ 
satisfies the $\mathsf{SLO}$ principle and is well-founded. 
\end{Thm} 
\begin{proof} 
The restriction of the Wadge quasi-order to subsets of the Baire space in $\mathcal{C}$ satisfies the $\mathsf{SLO}$ principle and is well-founded by our determinacy assumption (see \cite[Theorems 21.14 \& 21.15]{Kechris_Classical}). 
Hence it is sufficient for the first claim to show that $A\preceq B$ holds for all subsets $A$, $B$ of $\RR$ in $\mathcal{C}$ with $f^{-1}(A)\leq f^{-1}(B)$. To this end, we assume that $g$ 
is a continuous reduction of $f^{-1}(A)$ to $f^{-1}(B)$ on the Baire space. Since $f$ and $g$ are continuous and $f^{-1}$ is right-continuous, 
their composition $h=f\circ g \circ f^{-1}$ is right-continuous as well. 
Hence $A=f(g^{-1}(f^{-1}(B)))=h^{-1}[B]$ holds by the choice of $g$ and we thus have $A\preceq B$ as required. 

For the second claim, it is sufficient to show that $f^{-1}(A) < f^{-1}(B)$ holds for all subsets $A$, $B$ of $\RR$ with $A\prec B$.  
Assuming that $f^{-1}(A) \not\leq f^{-1}(B)$ holds towards a contradiction, we have  
$f^{-1}(\RR\setminus B)
\leq f^{-1}(A)$ by the $\mathsf{SLO}$ principle. 
Now $\mathbb{R} \setminus B \preceq A$ by the claim that is proved in the previous paragraph. 
Since the relation $\preceq$ is transitive and $A \prec B$ holds, we thus obtain $\mathbb{R} \setminus B \preceq B$ and $B\preceq \mathbb{R} \setminus B$ by passing to complements. Since $\mathbb{R} \setminus B \preceq A$ we now have 
$B\preceq A$, but this contradicts the assumption that the reducibility $A\prec B$ is strict. 
\end{proof} 


Besides the previous result, very little is known about the precise structure of this reducibility. 
In order to compare 
it with reducibilities that were studied by Motto Ros, Selivanov and the second author in \cite{MR3417077}, recall that a function 
is called $(\Gamma,\Delta)$-measurable for classes $\Gamma$ and $\Delta$ if the preimage 
of any set 
in $\Delta$ is in $\Gamma$.


\begin{Observation} 
It can be easily checked that the bijection $f\colon {}^\omega\omega\rightarrow \RR$ that is defined in the proof of Fact \ref{continuous bijection} is both $(\mathbf{\Sigma}^0_2,\mathbf{\Sigma}^0_1)$- and $(\mathbf{\Sigma}^0_3,\mathbf{\Sigma}^0_3)$-measurable. 
Therefore the $\mathsf{SLO}$ principle is still valid and the reducibility is well-founded if these conditions are required for reductions in addition to the previous ones. 
However, \cite[Corollary 5.21]{MR3417077} shows that the reducibility with respect to $(\mathbf{\Sigma}^0_2,\mathbf{\Sigma}^0_2)$-measurable functions on $\RR$ does not satisfy the $\mathsf{SLO}$ principle. 
\end{Observation} 

Finally, we would like to mention that Pequignot \cite{MR3372615} introduced a different notion of reducibility for arbitrary Polish spaces for which the $\mathsf{SLO}$ principle holds.

\section{Open questions}

We conclude with some open questions. We first ask wether the complexity of the sets that are constructed in Theorem \ref{embed Pomega mod fin} to embed $\mathcal{P}(\omega)/\mathsf{fin}$ is optimal. 

\begin{Question} 
Does Theorem \ref{embed Pomega mod fin} hold for the class of sets that are simultaneously $D_2({\bf \Sigma}^0_1)$ and $\check{D}_2({\bf \Sigma}^0_1)$? 
\end{Question} 

Moreover, we ask for complete descriptions of certain restricted classes of sets equipped with the Wadge quasi-order. 

\begin{Question} 
Can we give combinatorial descriptions of the following classes of sets equipped with the Wadge quasi-order? 
\begin{enumerate-(a)} 
\item 
Countable sets 
\item 
Countable unions of intervals 
\item 
Sets that are reducible to $\QQ$ 
\end{enumerate-(a)} 
\end{Question}

Since Theorem \ref{minimal set} shows that there is a minimal set below $\QQ$, this suggests to ask for a description of the well-founded part of the Wadge order and in particular to ask the next two questions. 

\begin{Question} 
Are there any sets besides the unique minimal set in the well-founded part of the restriction of the Wadge quasi-order to the sets that are reducible to $\QQ$? 
\end{Question}

\begin{Question} 
Are there infinitely many minimal Borel sets? 
\end{Question}

Moreover, the existence of gaps which is proved in 
Theorem \ref{gaps} suggests to ask whether there are types of gaps in the Wadge order that do not exist in $\mathcal{P}(\omega)/\mathsf{fin}$ and in particular the following question. 

\begin{Question} 
Is there an $(\omega,\omega)$-gap in the Wadge quasi-order? 
\end{Question}

We can further  ask about reductions that are defined with respect to more restricted classes of functions such as the following. 

\begin{Question} 
Does Theorem \ref{embed Pomega mod fin} hold for the classes of Lipschitz and $C^\infty$-functions? 
\end{Question}

\bibliographystyle{plain}
\bibliography{myreference}

\end{document}